\documentclass[a4paper, 10pt]{amsart}

\usepackage[square, numbers, comma]{natbib} % Use the natbib reference package - read up on this to edit the reference style; if you want text (e.g. Smith et al., 2012) for the in-text references (instead of numbers), remove 'numbers'. Add sort&compress to sort references the way they appear in the bibliography, and compress references which are close in numbers.

\usepackage[usenames,dvipsnames,svgnames,table]{xcolor}
\usepackage{amsmath,amsthm,amssymb,amssymb,esint,verbatim,tabularx,graphicx}
\usepackage{fancyhdr}
\usepackage{enumerate}
\usepackage{amsmath}
\usepackage{fancyhdr}
\usepackage{epic}
\usepackage{pgf,tikz}
\usetikzlibrary{arrows}

\usepackage[utf8]{inputenc}
\usepackage{color}
\usepackage{hyperref}
\usepackage{verbatim}
\usepackage{pdfpages}

\hypersetup{urlcolor=blue, colorlinks=true} % Colors hyperlinks in blue - change to black if annoying

\newtheorem{theorem}{Theorem}[section]

\newtheorem{lemma}[theorem]{Lemma}
\newtheorem{proposition}[theorem]{Proposition}
\newtheorem{corollary}[theorem]{Corollary}
\newtheorem{definition}{Definition}[section]

\theoremstyle{remark}
\newtheorem{remark}[theorem]{Remark}
\theoremstyle{definition}

\numberwithin{equation}{section}

\newcommand{\R}{\ensuremath{\mathbb{R}}}
\newcommand{\N}{\ensuremath{\mathbb{N}}}

\newcommand{\Z}{\ensuremath{\mathbb{Z}}}
\newcommand{\J}{\ensuremath{\mathbb{J}}}

\newcommand{\vertiii}[1]{{\left\vert\kern-0.25ex\left\vert\kern-0.25ex\left\vert #1 
    \right\vert\kern-0.25ex\right\vert\kern-0.25ex\right\vert}}

\newcommand{\Levy}{\ensuremath{\mathcal{L}}}

\newcommand{\Operator}{\ensuremath{\mathfrak{L}^{\sigma,\mu}}}

\newcommand{\Levymu}{\ensuremath{\mathcal{L}^\mu}}
\newcommand{\Levynu}{\ensuremath{\mathcal{L}^\nu}}

\newcommand{\Lvar}{L_{\varphi}}

\newcommand{\Te}{\ensuremath{T^{\textup{exp}}}}
\newcommand{\Ti}{\ensuremath{T^{\textup{imp}}}}

\newcommand{\W}{\mathcal{W}}

\newcommand{\dd}{\,\mathrm{d}}

\newcommand{\dell}{\partial}
\newcommand{\indikator}{\mathbf{1}_{|z|\leq 1}}
\newcommand{\indik}{\mathbf{1}}

\newcommand{\e}{\text{e}}

\DeclareMathOperator*{\esssup}{ess \, sup}

\newcommand{\Ut}{\widetilde{U}}
\newcommand{\Vt}{\widetilde{V}}

\newcommand{\Grid}{\mathcal{G}_h}
\newcommand{\GridT}{\mathcal{T}_{\Delta t}^T}
\newcommand{\ol}{\overline}

\begin{document}

%\title[Numerical analysis of equations of porous medium type]{Numerical methods and analysis for nonlocal (and local) equations of porous medium type. \\ Part I: Theoretical aspects}

\title[Numerical analysis of equations of porous medium
  type]{Robust numerical methods for nonlocal (and local)
  equations of porous medium type.\\ Part I: Theory}

%  Numerical methods and analysis for nonlocal (and local) equations of porous medium type. Part I: Theoretical analysis}

\author[F.~del~Teso]{F\'elix del Teso}
\address[F. del Teso]{Basque Center for Applied Mathematics (BCAM)\\
Bilbao, Spain} 
\email[]{fdelteso@bcamath.org}
\urladdr{http://www.bcamath.org/es/people/fdelteso}

\author[J.~Endal]{J\o rgen Endal}
\address[J. Endal]{Department of Mathematical Sciences\\
Norwegian University of Science and Technology (NTNU)\\
N-7491 Trondheim, Norway} 
\email[]{jorgen.endal\@@{}ntnu.no}
\urladdr{http://folk.ntnu.no/jorgeen}

\author[E.~R.~Jakobsen]{Espen R. Jakobsen}
\address[E. R. Jakobsen]{Department of Mathematical Sciences\\
Norwegian University of Science and Technology (NTNU)\\
N-7491 Trondheim, Norway} 
\email[]{espen.jakobsen\@@{}ntnu.no}
\urladdr{http://folk.ntnu.no/erj}



%\author{\sc F\'elix del Teso, J\o rgen Endal and Espen R. Jakobsen}

%\maketitle

\begin{abstract}
  \noindent We develop a unified and easy to use framework to study robust fully discrete numerical methods for nonlinear degenerate diffusion equations 
$$
\partial_t u-\mathfrak{L}^{\sigma,\mu}[\varphi(u)]=f  \quad\quad\text{in}\quad\quad \R^N\times(0,T),
$$ 
where $\mathfrak{L}^{\sigma,\mu}$ is a general symmetric diffusion operator of
L\'evy type and $\varphi$ is merely continuous and
non-decreasing. We then use this theory to prove
  convergence for many different numerical schemes. In the nonlocal
  case most of the results are completely new. Our theory covers strongly degenerate Stefan
problems,  the full range of porous medium equations, and for the
first time for nonlocal problems, also fast diffusion
equations. Examples of diffusion operators $\mathfrak{L}^{\sigma,\mu}$ 
are the (fractional) Laplacians $\Delta$ and
$-(-\Delta)^{\frac\alpha2}$ for $\alpha\in(0,2)$, discrete operators,
and combinations.   
The observation that monotone finite difference operators are nonlocal L\'evy
operators, allows us to give
a unified and compact {\em nonlocal} theory for both
local and nonlocal, linear and nonlinear diffusion equations.
The theory includes  stability, compactness, and convergence of the
methods under minimal assumptions -- including 
assumptions that lead to very irregular solutions. 
As a byproduct, we prove the new and general
existence result announced in \cite{DTEnJa17b}. We also
present some numerical tests, but extensive testing is
deferred to the companion paper \cite{DTEnJa18b} along with a more
detailed discussion of the numerical methods included in our theory.
\end{abstract}

%\vspace{0.5cm}
%
%\noindent\textbf{Keywords}: fully discrete, numerical schemes, convergence, uniqueness, distributional solutions, nonlinear degenerate diffusion, porous medium equation, fast diffusion equation, Stefan problem, fractional Laplacian, Laplacian, nonlocal operators, existence, a priori estimates 
%
%\vspace{0.5cm}
%
%\noindent 2010 {\sc Mathematics Subject Classification:} 35A02, 35B30, 35K65, 35D30, 35K65, 35R09, 35R11, 65R20, 76S05
%%26A33, %Fractional derivatives and integrals
%%35K65, %Parabolic partial differential equations of degenerate type
%%76S05, %Flows in porous media; filtration;
%%34A08, %Fractional differential equations
%%65R20  %Numerical Methods for Integral equations
%
%
%\vspace{0.5cm}
%
%\noindent \textbf{Addresses:} {\tt }\\
%F\'elix del Teso, {\tt felix.delteso@ntnu.no},\\
%J\o rgen Endal, {\tt jorgen.endal@ntnu.no},\\
%Espen R. Jakobsen, {\tt espen.jakobsen@ntnu.no},\\
%\vspace{-3mm}
%
%\noindent Department of Mathematical Sciences, NTNU Norwegian University of Science and Technology,\\
%NO-7491 Trondheim, Norway
%%Diana Stan, {\tt diana.stan@uam.es}, \\F\'{e}lix del Teso, {\tt felix.delteso@uam.es},
%%\\and Juan Luis V{\'a}zquez, {\tt juanluis.vazquez@uam.es},\\
%%Departamento de Matem\'{a}ticas, Universidad
%%Aut\'{o}noma de Madrid, \\
%%Campus de Cantoblanco, 28049 Madrid, Spain

\maketitle

%\newpage
%\small
\tableofcontents
%\newpage

\section{Introduction}
We develop a unified and easy to use framework for monotone schemes of
finite difference type for a large class of possibly degenerate,
nonlinear, and nonlocal diffusion equations of porous medium type. We
then use this 
theory to prove stability, compactness, and convergence for many
different robust schemes. In the nonlocal case most of the results
are completely new. The equation we study is
\begin{equation}\label{E}
\begin{cases}
\dell_tu-\Operator[\varphi(u)]=f \qquad\qquad&\text{in}\qquad Q_T:=\R^N\times(0,T),\\
u(x,0)=u_0(x) \qquad\qquad&\text{on}\qquad \R^N,
\end{cases}
\end{equation}
where $u$ is the solution, $\varphi$ is a merely
continuous and nondecreasing function, $f=f(x,t)$ some right-hand side, and $T>0$. The diffusion operator $\Operator$ is given~as 
\begin{equation}\label{defbothOp}
\Operator:=L^\sigma+\Levymu
\end{equation} 
with local and nonlocal (anomalous) parts, 
\begin{align}\label{deflocaldiffusion}
  L^\sigma[\psi](x)&:=\text{tr}\big(\sigma\sigma^TD^2\psi(x)\big),\\
\label{deflevy}
\Levy ^\mu [\psi](x)&:=\int_{\R^N\setminus\{0\} } \big(\psi(x+z)-\psi(x)-z\cdot D\psi(x)\indikator\big) \dd\mu(z),
\end{align}
%for
where  $\psi \in C_\textup{c}^2$, $\sigma=(\sigma_1,....,\sigma_P)\in\R^{N\times P}$ for $P\in \N$
and $\sigma_i\in \R^N$, $D$ and $D^2$ are the gradient and Hessian,
$\indikator$ is a characteristic function, and $\mu$ is a nonnegative
symmetric Radon measure. 

The assumptions we impose on $\Operator$ and $\varphi$ are so mild
that many different problems can be written in the form
\eqref{E}. The assumptions on $\varphi$ allow strongly degenerate
Stefan type problems and the full range of porous
medium and fast diffusion equations to be covered by \eqref{E}. In the
first case e.g. $\varphi(u)=\max(0,au-b)$ for $a\geq0$ and $b\in\R$
and in the second $\varphi(u)=u|u|^{m-1}$ for any $m\geq0$.
Some physical phenomena that can be modelled by
\eqref{E} are flow in a porous medium (oil, gas,
groundwater), nonlinear heat transfer, phase transition in matter,
and population dynamics. For more 
information and examples, we refer to 
Chapters 2 and 21 in \cite{Vaz07} for local problems and to
\cite{Woy01,MeKl04,Caf12,Vaz12} for nonlocal problems. 

One important contribution of this paper is that we allow for a very
large class of diffusion operators  $\Operator$. This class coincides
with the generators of the {\em symmetric} L\'evy processes. Examples are
 Brownian motion, $\alpha$-stable, relativistic, CGMY, and compound Poisson 
 processes \cite{Ber96,Sch03,App09}, and the generators include
 the classical and fractional Laplacians $\Delta$ and 
 $-(-\Delta)^{\frac{\alpha}{2}}$, $\alpha\in(0,2)$ (where
 $\dd\mu(z)=c_{N,\alpha}\frac{\dd z}{|z|^{N+\alpha}}$),
 relativistic Schr\"odinger operators $m^\alpha 
 I-(m^2I-\Delta)^{\frac{\alpha}{2}}$, and surprisingly, also
 monotone numerical discretizations of $\Operator$. % (see also
 % \cite{DTEnJa17a}).
 Since $\sigma$ and
 $\mu$ may be degenerate or even identically zero, problem \eqref{E}
 can be purely local, purely nonlocal, or a combination. % of the two. 
 
Nonstandard and novel ideas on numerical methods for
 \eqref{E} and their analysis
 are presented in this paper. We will strongly
 use the fact that our (large) class of diffusion
 operators contain many of its own monotone
 approximations. This important observation from \cite{DTEnJa17a} is used to
 interpret  discretizations of $\Operator$ as nonlocal
 L\'evy operators $\Levy^\nu$ which again opens the door
 for powerful PDE techniques and a unified analysis of our schemes.
We consider % spatial
discretizations of $\Operator$ of the form
$$
\Levy^h[\psi](x)=\sum_{\beta\neq0}\left(\psi(x+z_{\beta})-\psi(x)\right)\omega_{\beta},
$$
or equivalently $\Levy^h=\Levy^{\nu}$ with $\nu:=\sum_{\beta\not=0}
(\delta_{z_\beta}+\delta_{z_{-\beta}})\omega_{\beta}$, where 
$\beta\in \Z^N$, the stencil points $z_{\beta}\in
\R^N\setminus \{0\}$, the weights $\omega_{\beta}\geq0$, and
$z_{-\beta}= -z_{\beta}$ and $\omega_{\beta}=\omega_{-\beta}$. These
discretizations are nonpositive in the sense that
$\Levy^h[\psi](x_0)\leq0$ for any maximum point $x_0$ of 
$\psi\in C_{\textup{c}}^\infty(\R^N)$, and as we will see, they include monotone
finite difference quadrature approximations of $\Operator$. 
Our numerical approximations of  \eqref{E} will then take the general form
\begin{equation*}
\begin{split}
U_\beta^j=U_\beta^{j-1}+\Delta t_j\big(\Levy_1^h[\varphi_1^h(U_{\cdot}^j)]_\beta+\Levy_2^h[\varphi_2^h(U_\cdot^{j-1})]_\beta+F^j_\beta\big)
\end{split}
\end{equation*}
where $U_\beta^j\approx u(x_\beta,t_j)$, $\Levy^h_i \approx
\Operator$, $\varphi^h_i\approx \varphi$, $F_\beta^j\approx
f(x_\beta,t_j)$  and $h$ and $\Delta t_j$ are the discretization parameters in space and time respectively.  By choosing $\varphi^h_1, \varphi^h_2,\Levy^h_1,
\Levy^h_2$ in certain ways, we can recover explicit, implicit,
$\theta$-methods, and various explicit-implicit methods. In a simple one
dimensional case,
\begin{align*}
  \dell_t u&=\varphi(u)_{xx}-(-\partial_x^2)^{\alpha/2}\varphi(u),
\end{align*}
an example of a discretization in our class is given by
\begin{align*}
U^j_m=U^{j-1}_m&+\frac{\Delta
  t}{h^2}\Big(\varphi(U^{j}_{m+1})-2\varphi(U^{j}_{m})+\varphi(U^{j}_{m-1})\Big)\\
&+\Delta
    t\sum_{k\neq0}\Big(\varphi(U^{j-1}_{m+k})-\varphi(U^{j-1}_{m})\Big)\int_{(k-\frac12)h}^{(k+\frac12)h}\frac{c_{1,\alpha}\dd z}{|z|^{1+\alpha}}.
\end{align*}
Our class of schemes include both well-known discretizations and
many discretizations that are new in context of \eqref{E}. These new
discretizations include 
higher order discretizations of the nonlocal operators,
explicit schemes 
for fast diffusions, and various explicit-implicit schemes.
See the discussion in Sections 
\ref{sec:mainresults} and \ref{sec:discNumExp} and especially the
companion paper \cite{DTEnJa18b} for more details.

One of the main contributions of this paper is to provide a uniform and
rigorous analysis of such numerical schemes in this very general
setting, a setting that covers local and nonlocal, linear and nonlinear,
non-degenerate and degenerate, and smooth and nonsmooth problems. This
novel analysis includes well-posedness, stability, equicontinuity,
compactness, and $L^p_{\textup{loc}}$-convergence results for the
schemes, results which are completely new in some local and most nonlocal cases.
Schemes that converge in such general circumstances are often said to be {\em
  robust}. 
%Consistent numerical schemes \bc (i.e. convergent schemes for smooth functions)  are not robust in general,
%i.e. they need not always converge \bc for weak solutions, or can even converge to false
%solutions.   
Numerical schemes that are formally consistent are not robust in this generality, i.e. they need not always converge for problems with nonsmooth solutions or can even converge to false solutions. Such issues are seen especially in nonlinear, degenerate and/or
low regularity problems. Our 
general results are therefore only possible because we have (i) 
identified a class of schemes with good properties (including monotonicity)
and (ii) developed the necessary mathematical techniques for
this general setting.

A novelty of our analysis is that we are able to present the
theory in a uniform, compact, and natural way. By interpreting
discrete operators as nonlocal L\'evy operators, and the schemes as
holding in every point in space, we can use PDE type techniques for the
analysis. This is possible because in recent papers
\cite{DTEnJa17a,DTEnJa17b} we have developed  a well-posedness theory for problem
\eqref{E} which in particular allows for the general class of
diffusion operators needed here. Moreover, the well-posedness holds for
merely bounded distributional or very weak solutions. The fact that we
can use such a weak notion of solution will simplify the analysis and
make it possible to do a global theory for all the different problems
\eqref{E} and schemes that we consider here. At this point the reader should
note that if \eqref{E} has more regular (bounded) solutions (weak,
strong, mild, or classical), then our results still apply because
these solutions will coincide with the (unique) distributional
solution.

The effect of the L\'evy operator interpretation of the discrete
operators is that part of our analysis is turned in to a study of
semidiscrete in time approximations of \eqref{E} (cf. \eqref{defNumSch}).
% To do so we 
% adapt arguments from \cite{DTEnJa17a}. But since we now work with time
% discrete equations, we have to prove well-posedness and all a priori
% estimates anew. Unlike \cite{DTEnJa17a} where many results were inherited 
% by approximation from other theories and papers, we now prove them
% from scratch because we have to and it is easier in most
% cases. The proofs of this 
% paper are therefore more natural, elementary, and self-contained.
A convergence result for these are then
obtained from a compactness argument:
% used e.g.~for
%scalar conservation laws.
%, see Chapter 3 in \cite{HoRi02}. 
We prove 
(i) uniform estimates in $L^1$ and $L^\infty$ and space/time
translation estimates in $L^1$/$L^1_{\textup{loc}}$, (ii) compactness
in $C([0,T];L^1_{\textup{loc}}(\R^N))$ via the Arzel\`a-Ascoli and
Kolmogorov-Riesz theorems, (iii) limits of convergent subsequences are
distributional solutions via stability results for \eqref{E}, and
finally (iv) full convergence of the numerical solutions by (ii),
(iii), and uniqueness for \eqref{E}.  The proofs of the various a
priori estimates are done from scratch using new, efficient, and nontrivial
approximation arguments for nonlinear nonlocal problems.

To complete our proofs, we also need to connect the
results for the 
semi-discrete scheme defined on the whole space with the fully discrete scheme
defined on a spatial grid. We observe here that this part is easy
for uniform grids where we prove an equivalence theorem under
natural assumptions on discrete operators: Piecewise constant
interpolants of solutions of the fully discrete scheme coincides with
solutions of the corresponding semi-discrete scheme with piecewise
constant initial data (see Proposition 
\ref{prop:equivalence}). Nonuniform grids is a very interesting case that
we leave for future work. 

The nonlocal approach presented in this paper gives a uniform way of
representing local, nonlocal and discrete problems, different schemes {\em and} equations; compact,
efficient, and easy to understand PDE type arguments that work for very
different problems and schemes; new convergence results for local and nonlocal
problems; and it is very natural since the difference quadrature approximations 
$\Levy^h$ are nonlocal operators of the form \eqref{deflevy}, even when
equation \eqref{E} is local.  

We also mention that a consequence of our convergence and compactness
theory is the existence of 
distributional solutions of the Cauchy problem \eqref{E}. 

\smallskip
\noindent {\bf Related work.}
In the  local linear case, when $\varphi(u)=u$ and $\mu\equiv0$ in
\eqref{E}, numerical methods and analysis can be found in
undergraduate text books. In the nonlinear case there is a very large
literature so we will focus only on some developments that are more
relevant to this paper. For porous medium nonlinearities
($\varphi(u)=u|u|^{m-1}$ with $m>1$), there are early results on
finite element and finite-difference interface tracking methods in
\cite{Ros83} and \cite{DBHo84} (see also \cite{Mon16}). There is
extensive theory for finite volume schemes, see \cite[Section
  4]{EyGaHe00} and references therein for equations with locally
Lipschitz $\varphi$. For finite element methods there is a number of
results, including results for fast diffusions ($m\in(0,1)$), Stefan
problems, convergence for strong and weak solutions, discontinuous
Galerkin methods, see
e.g. \cite{RuWa96, EbLi08,EmSi12, DuAlAn13,ZhWu09,NoVe88,MiSaSu05}. 
%% There is a number of results for
%% finite element methods: Methods that can handle also the fast
%% diffusion range ($m\in(0,1)$) 
%% convergence results for weak solutions \cite{EmSi12,},
%% schemes for more general porous medium like equations 
%% (including Stefan problems) \cite{NoVe88,MiSaSu05}. 
Note that the latter paper considers the general form of \eqref{E}
with $\Operator=\Delta$ and provides a convergence analysis in $L^1$
using nonlinear semi-group theory. A number of results on finite
difference methods for degenerate convection-diffusion equations also yield
results for \eqref{E} in special cases, see e.g. \cite{EvKa00,BuCoSe06,KaRiSt16,JePa17}. 
In particular the results of \cite{EvKa00,KaRiSt16} imply our
convergence results for a particular scheme when $\varphi$ is
locally Lipschitz,  $\Operator=\Delta$, and solutions have a certain
additional BV regularity. Finally, we mention very
general results on so-called gradient schemes
\cite{DrEyGaHe13,DrEyHe16,DL2019} for porous medium
equations or more general doubly or triply degenerate parabolic
equations.

In the nonlocal case, the literature is more recent and not so
extensive. For linear equations in the whole space, finite difference
methods have been studied in e.g. \cite{CoTa04,HuOb14,
  HuOb16,CiRoStToVa18}. An important but different line of research
concerns problems on bounded domains, see e.g. \cite{DEGu13, BoPa15,
  NoOtSa15, AcBo17, CuDTG-GPa18}. This direction  will not be
discussed further in this paper. 
%we refer somewhat arbitrarily to \cite{CoTa04,HuOb14, HuOb16, NoOtSa15}  and references therein. 
%Finite elements and bounded domains are  different research topics than the one presented in this paper, and we will therefore focus on references regarding finite differences in $\R^N$.
Some early numerical results for nonlocal problems came for finite difference quadrature
schemes for Bellman equations and fractional conservation laws, see
\cite{JaKaLC08,CaJa09,BiJaKa10} and \cite{Dro10}. 
For the latter case discontinuous Galerkin and spectral methods were later studied in \cite{CiJaKa11,CiJa13,XuHe14}. 
The first results that include nonlinear nonlocal versions of \eqref{E} was
probably given in \cite{CiJa11}. Here convergence of finite
difference quadrature schemes was proven for a convection-diffusion
equation. This result is extended to more general equations and error
estimates in \cite{CiJa14} and a higher order discretization in
\cite{DrJa14}. In some cases our convergence results follow from these
results (for two particular schemes, $\sigma=0$, and $\varphi$ locally
Lipschitz). However, the analysis there is different and more 
complicated since it involves entropy solutions and Kru\v{z}kov doubling
of variables arguments.

In the purely parabolic case \eqref{E}, the behaviour of the solutions
and the underlying theory is different from the convection-diffusion
case (especially so in the nonlocal case, see e.g. \cite{DPQuRoVa11, DPQuRoVa12, Vaz14,DPQuRo16,DPQuRoVa17} and \cite{DrIm06, ChCzSi10, AlAn10,CiJa11,AlCiJa14,IgSt18}). It is therefore important to develop numerical methods and
analysis that are specific for this setting. The first numerical results for Fractional Porous Medium Equations seem to be  \cite{DTe14, DTVa14}  which are based on the extension method \cite{CaSi07}. The present paper is another step in this direction, possibly the first not to use the extension method in this setting.  

\smallskip
\noindent{\bf Outline.} The assumptions, numerical schemes, and main
results are given in Section 2. In Section 3 we provide many
concrete examples of schemes that satisfy the assumptions of Section
2. We also show some numerical results for a nonlocal Stefan
problem with non-smooth solutions. The proofs of the main results are
given in Section 4, while some auxiliary results are proven in our
final section, Section 5.

\smallskip

In the companion paper \cite{DTEnJa18b} there is a more complete
discussion of the family of numerical methods. It includes more
discretizations of the operator $\Operator$, more schemes, and many
numerical examples. There we also provide proofs and explanations for
why the different schemes satisfy the (technical) assumptions of this paper.

%%%%%%%%%%%%%%%%%%%%%%%%%%%%%%%%%%%%%%%%%%%%%%%%%%%%
%%%%%%%%%%%%%%%%%%%%%NEW SECTION%%%%%%%%%%%%%%%%%%%%%%%
%%%%%%%%%%%%%%%%%%%%%%%%%%%%%%%%%%%%%%%%%%%%%%%%%%%%

\section{Main results}
\label{sec:mainresults}

The main results of this paper are presented in this section. They
include the definition of the numerical schemes, their consistency,
monotonicity, stability, and convergence of numerical solutions towards 
distributional solutions of the porous medium type equation \eqref{E}.

\subsection{Assumptions and preliminaries}
\label{sec:assumptionsprelim}

The assumptions on \eqref{E} are
\begin{align}
&\varphi:\R\to\R\text{ is nondecreasing and continuous};
\tag{$\textup{A}_\varphi$}&
\label{phias}\\
&f\text{ is measurable and }\int_0^T\big(\|f(\cdot,t)\|_{L^1(\R^N)}+\|f(\cdot,t)\|_{L^\infty(\R^N)}\big)\dd t<\infty;
\tag{$\textup{A}_f$}&
\label{gas}\\
&u_0\in L^1(\R^N)\cap L^{\infty}(\mathbb{R}^N);\text{ and}
\tag{$\textup{A}_{u_0}$}&
\label{u_0as}\\
&\label{muas}\tag{$\textup{A}_{\mu}$} \mu \text{ is a nonnegative symmetric Radon measure on
}\R^N\setminus\{0\}
\text{ satisfying}
\nonumber\\ 
&\quad\int_{|z|\leq1}|z|^2\dd \mu(z)+\int_{|z|>1}1\dd
\mu(z)<\infty\nonumber.
\end{align}
Sometimes we will need stronger assumptions than \eqref{phias} and
\eqref{muas}:
\begin{align}
&\label{philipas}\tag{$\textup{Lip}_\varphi$} \varphi:\R\to\R\text{ is nondecreasing and locally Lipschitz; and}\\
&\label{nuas}\tag{$\textup{A}_{\nu}$} \nu \text{ is a nonnegative symmetric Radon measure satisfying } \nu(\R^N)<\infty. \nonumber
\end{align}

\begin{remark}\label{addingconstants}
\begin{enumerate}[{\rm (a)}]
\item
  Without loss of generality, we can assume $\varphi(0) = 0$
  (replace $\varphi(u)$ by $\varphi(u)-\varphi(0)$), and when
  \eqref{philipas} holds, that $\varphi$ is globally Lipschitz (since
  $u$ is bounded). In the latter case we let $L_\varphi$ denote the
  Lipschitz constant.
  \smallskip
\item Under assumption \eqref{muas}, for any $p\in[1,\infty]$ and any
  $\psi\in C_\textup{c}^\infty(\R^N)$,
\begin{equation}\label{eq:WPL}\|\Operator[\psi]\|_{L^p}\leq c\|D^2\psi\|_{L^p}
  \Big(|\sigma|^2+\int_{|z|\leq1}|z|^2\dd\mu(z)\Big)+ 2\|\psi\|_{L^p}
  \int_{|z|>1}\dd\mu(z).\end{equation}
 \item Assumption \eqref{gas} is equivalent to requiring 
$
f\in L^1(0,T;L^1(\R^N)\cap L^\infty(\R^N))
$, an iterated $L^{P}$-space as in e.g.
\cite{BePa61}. Note that $L^1(0,T;L^1(\R^N))=L^1(Q_T)$.
\end{enumerate}
\end{remark}

\begin{definition}[Distributional solution]\label{distsol} 
Let  $u_0\in L_\textup{loc}^1(\R^N)$ and  $f\in
L_\textup{loc}^1(Q_T)$. Then $u\in L_\textup{loc}^1(Q_T)$ is a
distributional (or very weak) solution of \eqref{E} if for all $\psi\in
C_\textup{c}^\infty(\R^N\times[0,T))$, $\varphi(u)\Operator[\psi]\in
L^1(Q_T)$ and
\begin{align}\label{def_eq}
\int_{0}^{T}\int_{\R^N} \big(u\dell_t\psi+ \varphi(u)\Operator[\psi]+f\psi\big)\dd x \dd t+\int_{\R^N}u_0(x)\psi(x,0)\dd x=0
\end{align}
\end{definition}
Note that  $\varphi(u)\Operator[\psi]\in
L^1$ if e.g.~$u\in L^\infty$ and $\varphi$ continuous. Distributional solutions are unique in $L^1\cap L^\infty$.
\begin{theorem}[Theorem 3.1 \cite{DTEnJa17b}]\label{unique}
  Assume \eqref{phias}, \eqref{gas}, \eqref{u_0as}, and
  \eqref{muas}. Then there is at most one
  distributional solution $u$ of \eqref{E} such that $u\in L^1(Q_T)\cap
L^\infty(Q_T)$. 
  \end{theorem}

\subsection{Numerical schemes without spatial grids}
\label{num_wog}

Let $\mathcal{T}_{\Delta t}^T=\{t_j\}_{j=0}^{J}$ be a 
nonuniform grid in time such that $0=t_0<t_1<\ldots <t_J=T$. Let
$\J:=\{1,\ldots,J\}$, and denote time steps by
\begin{equation}\label{def:timesteps}
\Delta t_j=t_j-t_{j-1} \quad \textup{for every} \quad j\in \J, \quad \textup{and}  \quad \Delta t=\max_{j\in \J}\{\Delta t_j\}.
\end{equation}
For $j\in\J$ and $x\in\R^N$, we define
\begin{equation}\label{sourceNumTime}
F(x,t_j):=F^j(x)=\frac{1}{ \Delta t_j}\int_{t_j-\Delta t_j}^{t_j} f(x,t) \dd t,
\end{equation}
and we define our time discretized scheme, for $h>0$, as
\begin{equation}\label{defNumSch}
\begin{cases}
U_h^j(x)=U_h^{j-1}(x)+\Delta t_j\Big(\Levy_1^{h}[\varphi_1^h(U_h^j)](x)+\Levy_2^{h}[\varphi_2^h(U_h^{j-1})](x)+F^j(x)\Big)\\[0.2cm]
U_h^0(x)=u_0(x)
\end{cases}
\end{equation}
where, formally, $U_h^j(x)\approx u(x,t_j)$, $\frac{U_h^j(x)-U_h^{j-1}(x)}{\Delta t_j}\approx \dell_tu(x,t_j)$, and 
$$
\Levy_1^{h}[\varphi_1^h(U_h^j)](x)+\Levy_2^{h}[\varphi_2^h(U_h^{j-1})](x)\approx \Operator[\varphi(u)](x,t_j).
$$

Typically $\varphi_1^h=\varphi=\varphi_2^h$, but when
   $\varphi$ is not Lipschitz, we have to approximate it by a
   Lipschitz $\varphi_2^h$ to get a monotone explicit method
   \cite{DTEnJa18b}.
   Let $\varphi_1^h=\varphi=\varphi_2^h$. Depending on the choice of
   $\Levy_1^{h}$ and $\Levy_2^{h}$, we can then get many different
   schemes:
\begin{enumerate}[{\rm (1)}]
 \item Discretizing separately the different parts of the operator
   $$\Operator=L^\sigma+\mathcal L^\mu_{\text{sing}}+\mathcal
   L^\mu_{\text{bnd}},$$
   e.g. the local, singular nonlocal, and bounded nonlocal
   parts, corresponds to different choices for  $\Levy_1^h$ and
   $\Levy_2^h$. Typical choices here are finite difference and numerical
   quadrature methods, see Section \ref{sec:discNumExp} for  several 
   examples.\smallskip 
\item Explicit schemes ($\theta=0$),
   implicit schemes ($\theta=1$), or combinations like Crank-Nicholson
   ($\theta=\frac12$), follow by the choices
   $$\Levy_1^{h}=\theta\Levy^{h}\qquad\text{and}\qquad
   \Levy_2^{h}=(1-\theta)\Levy^{h}.$$
 \item Combinations of type (1) and (2) schemes,
   e.g. implicit discretization of the unbounded part of $\Operator$
   and explicit discretization of the bounded part.
\end{enumerate}
Finally, we mention that our schemes and results may easily be
extended to handle any finite number of  $\varphi_1^h,\dots,\varphi_m^h$ and $\Levy_1^h,\dots,\Levy_m^h$.

\begin{definition}[Consistency]\label{schCon} 
We say that the scheme \eqref{defNumSch}  is \textup{\textbf{consistent}} if, for $\varphi_1,\varphi_2,\varphi$ satisfying \eqref{phias}, $\mu$
 \eqref{muas}, and $\Operator_1$, $\Operator_2$, $\Operator$ of the form \eqref{defbothOp}--\eqref{deflevy}, 
\begin{enumerate}[{\rm (i)}]
\item\label{schCon:Eq} $\Operator_1[\varphi_1(\phi)]+\Operator_2[\varphi_2(\phi)]=\Operator[\varphi(\phi)]$ in $\mathcal{D}'(Q_T)$ for $\phi\in L^1(\R^N)\cap L^\infty(\R^N)$,
\item\label{schCon:Op} for all $ \psi \in C_\textup{c}^\infty(\R^N)$ and some $k_1,k_2\geq0$,
$$
\|\Levy_i^{h}[\psi]-\Operator_i[\psi]\|_{L^1(\R^N)}\leq \|\psi\|_{W^{k_i,1}(\R^N)}o_h(1)\stackrel{h\to0^+}{\longrightarrow}0 \qquad\textup{for}\qquad i=1,2,
$$
\item\label{schCon:Non} $\varphi_1^h,\varphi_{2}^h\to \varphi_1,\varphi_2$ locally uniformly as $h\to0^+$.
\end{enumerate}
\end{definition}

\begin{remark}
In view of step 4) in the proof of Lemma \ref{lem:conv}, condition
\eqref{schCon:Op} can be replaced by the following
more general consistency condition
$$
\|\Levy_i^{h}[\psi(\cdot,t)]-\Operator_i[\psi(\cdot,t)]\|_{C([0,T];L^1(\R^N))}\stackrel{h\to0^+}{\longrightarrow}0\qquad\textup{for all}\qquad \psi\in C_\textup{c}^\infty(Q_T)
$$ 
and for $i=1,2$. This concept of consistency holds for all the discretizations we are considering; see also the companion paper \cite{DTEnJa18b}.
\end{remark}

%\begin{enumerate}[{\rm (i)}]
%\item $\Levy_1^{h}[\psi],\Levy_2^{h}[\psi]\to \Operator_1[\psi],\Operator_2[\psi]$ in $L^1(\R^N)$ as $h \to0^+$ for
%  all $ \psi \in C_\textup{c}^\infty(\R^N)$,  
%\medskip
%\item $\varphi_1^h,\varphi_{2}^h\to \varphi_1,\varphi_2$ locally uniformly as $h\to0^+$,
%  \medskip
%\item $\Operator_1[\varphi_1(\phi)]+\Operator_2[\varphi_2(\phi)]=\Operator[\varphi(\phi)]$ in $\mathcal{D}'(Q_T)$ for $\phi\in L^1(\R^N)\cap L^\infty(\R^N)$,
%\end{enumerate}
%\smallskip
%for $\varphi_1,\varphi_2,\varphi$ satisfying \eqref{phias}, $\mu$
% \eqref{muas}, and $\Operator_1$, $\Operator_2$, $\Operator$ of the form \eqref{defbothOp}--\eqref{deflevy}. 
%\end{definition}

We will focus on discrete operators $\Levy_i^h$, $i=1,2$ in the following class:

\begin{definition}\label{LevyInTheClass}
  An operator $\Levy$ is said to be
  \smallskip
\begin{enumerate}[{\rm (i)}]
\item {\bf in the class \eqref{nuas}} if $
  \Levy=\Levy^{\nu}$ for a measure $\nu$ satisfying \eqref{nuas}; and
  \medskip
\item {\bf discrete} if
\[
\nu=\sum_{\beta\not=0} (\delta_{z_\beta}+\delta_{z_{-\beta}}) \omega_\beta
\]
for $z_\beta=-z_{-\beta}\in\R^N$ and $\omega_\beta=\omega_{-\beta}\in
\R_+$ such that $\sum_{\beta\not=0}\omega_\beta<\infty$.
\medskip 
\item $\mathcal S=\{z_\beta\}_{\beta}$ is called the {\bf stencil} and
  $\{\omega_\beta\}_\beta$ the {\bf weights} of the discretization.
\end{enumerate}
\end{definition}

All operators in the class \eqref{nuas} are nonpositive operators, in particular they are integral or quadrature operators with positive weights.
The results presented in this section hold for any operator in the
class \eqref{nuas}. However, in practice, when dealing with numerical
schemes, the operators will additionally be discrete. Moreover, when the scheme \eqref{defNumSch} has an explicit part, that is
$\nu_2^h$ and $\varphi_2^h$ are not simultaneously zero, we need to assume that $\varphi_2^h$
satisfies \eqref{philipas} and impose the following CFL-type condition
to have a monotone scheme:  
\begin{equation}\label{CFL}\tag{CFL}
\Delta tL_{\varphi_2^h} \nu_{2}^h(\R^N)\leq1,
\end{equation}
where we recall that $L_{\varphi_2^h}$ is the Lipschitz constant
of $\varphi_2^h$ (see Remark \ref{addingconstants}). Note that this condition is always satisfied for an
implicit method where $\nu_2^h\equiv0$.
The typical
assumptions on the scheme \eqref{defNumSch} are then:
\begin{equation}\label{asNumSch}\tag{$\textup{A}_{\textup{NS}}$}
\begin{cases}
\text{ $\Levy_1^h,\Levy_2^h$ are in the class \eqref{nuas} with respective measures $\nu_1^h,\nu_2^h$,}\\[0.2cm]
\text{$\varphi_1^h,\varphi_2^h$ satisfy \eqref{phias}, \eqref{philipas} respectively, and}\\[0.2cm]
\text{$\Delta t>0$ is such that \eqref{CFL} holds.}
\end{cases}
\end{equation}

\begin{theorem}[Existence and uniqueness]\label{thm:existuniqueNumSch}
Assume \eqref{asNumSch}, \eqref{gas}, and \eqref{u_0as}.
Then there exists a unique a.e.-solution $U_h^j\in L^1(\R^N)\cap L^\infty(\R^N)$ of the scheme \eqref{defNumSch}.
\end{theorem}

\begin{remark}
Since $U_h^j$ is a Lebesgue measurable function, it is not immediately clear that $\varphi_1^h(U_h^j),\varphi_2^h(U_h^{j-1})$ are $\nu_1^h,\nu_2^h$-measurable and $\Levy_1^h[\varphi_1^h(U_h^j)],\Levy_2^h[\varphi_2^h(U_h^{j-1})]$ are pointwisely well-defined. However, we could simply consider a Borel measurable a.e. representative of $U_h^j$; see also Remark 2.1 (1) and (2) in \cite{AlCiJa12} for a discussion.
\end{remark}

\begin{theorem}[A priori estimates]\label{thm:propertiesscheme}
Assume \eqref{asNumSch}, \eqref{gas}, and \eqref{u_0as}. Let
$U^j_h,V^j_h$ be solutions of the scheme \eqref{defNumSch} with data
$u_0,v_0$ and $f,g$. Then:
\smallskip
\begin{enumerate}[{\rm (a)}]
\item \textup{(Monotonicity)}  \label{thm:propertiesscheme:monotone} If $u_{0}(x)\leq v_{0}(x)$ and $f(x,t)\leq g(x,t)$, then $U_h^{j}(x)\leq V_h^{j}(x)$.
\medskip
\item  \textup{($L^1$-stability)} \label{thm:propertiesscheme:intstable}
  $\|U_h^j\|_{L^1(\R^N)}\leq
  \|u_0\|_{L^1(\R^N)}+\int_{0}^{t_{j}}\|f(\cdot,\tau)\|_{L^1(\R^N)}\dd
  \tau$.
  \medskip
\item  \textup{($L^\infty$-stability)} \label{thm:propertiesscheme:bndstable} $\|U_h^j\|_{L^\infty(\R^N)}\leq \|u_{0}\|_{L^\infty(\R^N)}+\int_{0}^{t_{j}}\|f(\cdot,\tau)\|_{L^\infty(\R^N)}\dd \tau$.\medskip
\item\textup{(Conservativity)}\label{thm:propertiesscheme:conservative} If $\varphi_1^h$ additionally satisfies \eqref{philipas}, 
$$
\int_{\R^N}U_h^j(x)\dd x=\int_{\R^N}u_0(x)\dd x+\int_0^{t_j}\int_{\R^N}f(x,\tau)\dd x \dd \tau.
$$
\end{enumerate}
\end{theorem}

\begin{remark}
By \eqref{thm:propertiesscheme:intstable}, \eqref{thm:propertiesscheme:bndstable},
and interpolation, the scheme is $L^p$-stable for $p\in[1,\infty]$.
\end{remark}

The scheme is also $L^1$-contractive and
  equicontinuous in time. Combined, these two results imply time-space
  equicontinuity and compactness of the scheme, a key step in our proof
  of convergence.

\begin{theorem}[$L^1$-contraction]\label{thm:propertiesscheme:contractive}
Under the assumptions of Theorem \ref{thm:propertiesscheme}, 
$$\int_{\R^N}(U_h^j-V_h^j)^+(x)\dd x
\leq \int_{\R^N}(u_{0}-v_{0})^+(x)\dd x+\int_0^{t_j}\int_{\R^N}(f-g)^+(x,\tau)\dd x\dd \tau.$$
\end{theorem}
For the equicontinuity in space and time we need a modulus of continuity:

\begin{align}\label{timeRegMod}
\Lambda_{K}(\zeta)&:= 2\, \lambda_{u_0,f}(\zeta^\frac{1}{3})+C_K(\zeta^{\frac{1}{3}}+\zeta),
\end{align}
  
where
\begin{align}
 &\lambda_{u_0,f}(\zeta):=
                        \sup_{|\xi|\leq\zeta}\Big(\|u_0-u_0(\cdot+\xi)\|_{L^1(\R^N)}+\|f-f(\cdot+\xi,\cdot)\|_{L^1(Q_T)}\Big),\label{omegau0f}\\ 
  &C_K:=c 
|K|\sup_{\scriptsize\begin{array}{c}h<1, \\  i=1,2\end{array}}\!\!\Big(1+\sup_{|\zeta|\leq
  M_{u_0,f}}\!\!\!|\varphi_i^h(\zeta)|\Big)\Big(1+\int_{|z|>0} |z|^2 \wedge1\ \dd \nu_i^h(z)\Big)\label{timeRegConst}
\end{align}
for some constant $c\geq1$, $a\wedge b:=\min\{a,b\}$, $K\subset \R^N$
compact with Lebesgue measure $|K|$, and
$M_{u_0,f}:=\|u_0\|_{L^\infty(\R^N)}+\int_0^T\|f(\cdot,\tau)\|_{L^\infty(\R^N)}\dd
\tau$. 
In view of \eqref{timeRegConst}, we also need to assume a uniform L\'evy condition on the approximations, 
\begin{equation}\tag{$\textup{A}_{\nu^h}$}\label{AsUnifLe}
\sup_{\scriptsize\begin{array}{c}h<1, \\  i=1,2\end{array}}\int_{|z|>0} |z|^2 \wedge1\ \dd \nu_i^h(z)<+\infty.
\end{equation}

\begin{remark}
Condition \eqref{AsUnifLe} is in general very easy to check. For
example it follows from pointwise consistency of $\mathcal L_i^h$
as we will see in \cite{DTEnJa18b}.
\end{remark}

\begin{theorem}[Equicontinuity in time]\label{thm:timereg}
Assume \eqref{gas} and \eqref{u_0as}, and let \eqref{defNumSch} be a consistent scheme satisfying \eqref{asNumSch} and \eqref{AsUnifLe}. Then,
for all $j,k\in \J$ such that $j-k\geq0$ and all compact sets $K\subset \R^N$, 
\begin{equation*}
\begin{split}
&\|U_h^j-U_h^{j-k}\|_{L^1(K)}\leq
\Lambda_K(t_j-t_{j-k})+|K|\int_{t_{j-k}}^{t_j}\|f(\cdot,\tau)\|_{L^\infty(\R^N)}\dd
\tau,
\end{split}
\end{equation*}
where $\Lambda_K$ is defined in \eqref{timeRegMod}.
\end{theorem}
 
%\begin{remark}\label{F-rem}
%Let $F^l$ be defined in \eqref{sourceNumTime}. In the proof of Theorem \ref{thm:timereg} we use that
%\begin{align*}
%\sum_{l=j-k+1}^{j}\|F^l\|_{L^\infty(\R^N)}\Delta t_l \leq\int_{t_{j-k}}^{t_j}\|f(\cdot,\tau)\|_{L^\infty(\R^N)}\dd
%    \tau.
%\end{align*}
%\end{remark}

The main result regarding convergence of numerical schemes without spatial grids will be presented in a continuous in time and space framework. For that reason, let us define the piecewise linear time interpolant $\Ut_{h}$, for $(x,t)\in Q_T$, as

\begin{equation}\label{xt-interp}
\begin{split}
\Ut_{h}(x,t)&:=U_h^0(x)\indik_{\{t_0\}}(t)\\
&\quad+\sum_{j=1}^J
\indik_{(t_{j-1},t_j]}(t) \Big(U_h^{j-1}(x)+\frac{t-t_{j-1}}{t_j-t_{j-1}}\big(U_h^j(x)-U_h^{j-1}(x)\big)\Big).
\end{split}
\end{equation}

\begin{theorem}[Convergence]\label{thm:conv}
Assume \eqref{gas}, \eqref{u_0as}, $\Delta t=o_h(1)$, and for all
$h>0$, let $U^j_h$ be the solution of a consistent scheme \eqref{defNumSch} satisfying \eqref{asNumSch} and \eqref{AsUnifLe}. 
%, and let $\{\Ut_h\}_{h>0}$ be a sequence of interpolants  of
%$\{U_h^j\}_{h>0}$.
Then there exists a unique distributional solution $u\in L^1(Q_T)\cap
L^\infty(Q_T)\cap C([0,T];L_\textup{loc}^1(\R^N))$ of \eqref{E} and
\[
\Ut_{h}\to u \quad \text{in}  \quad C([0,T];L_\textup{loc}^1(\R^N)) \quad  \text{as} \quad h\to0^+.
\]
\end{theorem}

Convergence of subsequences follows from compactness and full
convergence follows from stability and uniqueness of the limit problem \eqref{E}.
The detailed proofs of Theorems \ref{thm:existuniqueNumSch}, \ref{thm:propertiesscheme}, and \ref{thm:propertiesscheme:contractive}--\ref{thm:conv} can be found in Sections \ref{sec:proofPropNumSch}--\ref{sec:proofConvNumSch}.

\begin{remark}\label{rem:RelationWithLitterature}
In this paper, we use \emph{piecewise linear interpolation} to
ensure that $\Ut_h$ belongs to $C([0,T];L_\textup{loc}^1(\R^N))$.
Moreover, we
obtain an equicontinuity result in time uniformly in $\Delta
t=o_h(1)$. Compactness and convergence then follows from 
Arzel\`a-Ascoli and Kolmogorov-Riesz type compactness results (see
e.g. \cite{DTEnJa19b}).

In most of the related literature piecewise
constant interpolation is used. In this case there is no convergence in
$C([0,T];L_\textup{loc}^1(\R^N))$, but one can use Kru\v{z}kov type
interpolation lemmas along with the Kolmogorov-Riesz compactness theorem to get convergence in $L_\textup{loc}^1(Q_T)$. Consult e.g. \cite{Kru69} for the vanishing viscosity limit of scalar conservation laws; \cite{KaRi01} for finite-difference approximations of convection-diffusion equations; \cite{AnGuWi04} for finite volume approximations of nonlinear elliptic-parabolic problems; and \cite{CiJa14} for finite volume approximations of nonlocal convection-diffusion equations. 
Yet another approach is discontinuous versions of the Arzel\`a-Ascoli compactness theorem (combined with Kolmogorov-Riesz) to get convergence in $L^\infty((0,T);L_\textup{loc}^1(\R^N))$; see the appendix of \cite{DrEy16}. 
\end{remark}

\subsection{Numerical schemes on uniform spatial grids}

To get computable schemes, we need to
introduce spatial grids. For simplicity we restrict to uniform grids.
Since our discrete operators have weights and stencils
not depending on the position $x$, all results then
become direct consequences of the results in Section \ref{num_wog}.

Let $h>0$, $R_h=h(-\frac{1}{2},\frac{1}{2}]^N$, and $\Grid$ be the
  uniform spatial grid
\[
\Grid:= h\Z^N=\{x_\beta:=h \beta: \beta\in \Z^N\}.
\]
Note that any discrete \eqref{nuas}-class operator $\Levy^h$ {\it with stencil $\mathcal{S}\subset \Grid$} is defined by
%\color{red}$\Levy^h:L^\infty(\R^N; \dd\nu^h) \normalcolor\to\R$ {\it with stencil $\mathcal
%S\subset \Grid$}, has a well-defined restriction $\Levy^h:
%\color{red}L^\infty(\Grid;\dd\nu^h) \normalcolor\to \R$ 
%defined by
\begin{equation*}
\Levy^h[\psi](x_\beta)=\Levy^h[\psi]_\beta=\sum_{\gamma\not=0}
(\psi(x_\beta+z_\gamma)-\psi(x_\beta)) \omega_{\gamma,h}\quad\text{for
all}\quad x_\beta\in\Grid
\end{equation*}
and all $\psi:\Grid\to \R$.
  Using such discrete operators, we get the following well-defined
  numerical discretization of \eqref{E} on the space-time grid
  $\Grid\times \GridT$, 
\begin{equation}\label{FullyDiscNumSch1}
U_\beta^j=U_\beta^{j-1}+\Delta t_j\big(\Levy_1^h[\varphi_1^h(U_{\cdot}^j)]_\beta+\Levy_2^h[\varphi_2^h(U_\cdot^{j-1})]_\beta+F^j_\beta\big),\quad \beta \in \Z^N,  j \in \J, \\
\end{equation}
where $U^0_\beta$ and $F_\beta^j$ are the cell averages of the
$L^1$- functions $u_0$ and $f$: 
\begin{equation}\label{averaged1}
U^0_\beta=\frac{1}{h^N}\int_{x_\beta+R_h}u_0(x)\dd x,\quad\  F^j_\beta=\frac{1}{h^N \Delta t_j}\int_{t_j-\Delta t_j}^{t_j}\int_{x_\beta+R_h}f(x,\tau)\dd x \dd \tau.
\end{equation}

The function $F=F_\beta^j$ and the solution $U=U_\beta^j$ are functions on $\Grid\times \GridT$, and we
define their piecewise constant interpolations in space as 
\begin{align}\label{x-interp}
\overline{U^j}(x):=\sum_{\beta\neq0} \indik_{x_\beta+R_h}(x)U_\beta^j
\qquad \textup{and} \qquad 
\overline{F^j}(x):=\sum_{\beta\neq0} \indik_{x_\beta+R_h}(x)F_\beta^j.
\end{align}
The next proposition shows that solutions of the scheme
\eqref{defNumSch} with piecewise constant initial data are solutions
of the fully discrete scheme
\eqref{FullyDiscNumSch1} and vice versa.

\begin{proposition}\label{prop:equivalence} 
Assume \eqref{gas}, \eqref{u_0as}, let
$U^0$, $F$ be defined by \eqref{averaged1} and 
$\overline{U^0}$, $\overline{F^j}$ by \eqref{x-interp}, and let
$\Levy_1^h$, $\Levy_2^h$ be class \eqref{nuas} discrete operators
with stencils $\mathcal S_1,\mathcal S_2\subset \Grid$.
\begin{enumerate} [{\rm (a)}]
\item If  $U^j=U^j(x)$ is an a.e. solution of \eqref{defNumSch} with data
  $\overline{U^0}$ and $\overline{F^j}$, then (a version of) $U^j$ is
  constant on the cells $x_\beta+R_h$ for all $\beta$,  and
  $U_\beta^j:=U^j(x_\beta)$ is a solution of \eqref{FullyDiscNumSch1}
  with data $U^0_\beta$ and $F_\beta^j$. 
\item If $U_\beta^j$ is a solution of \eqref{FullyDiscNumSch1} with
  data $U^0_\beta$ and $F_\beta^j$, then $\overline{U^j}(x)$ defined
  in \eqref{x-interp} is a piecewise constant solution of
  \eqref{defNumSch} with data $\overline{U^0}$ and $\overline{F^j}$. 
\end{enumerate}

\end{proposition}

In view of this result,  the scheme on the spatial grid
\eqref{FullyDiscNumSch1} will inherit the results for the scheme
\eqref{defNumSch} given in Theorems \ref{thm:existuniqueNumSch},
\ref{thm:propertiesscheme},
\ref{thm:propertiesscheme:contractive}--\ref{thm:conv}. 

%% \begin{theorem}[Existence and uniqueness]\label{thm:existuniqueNumSchFully}
%% Assume \eqref{asNumSch}, \eqref{gas}, and \eqref{u_0as}.
%% Then, there exists a unique a.e.-solution $U_\beta^j\in \ell^1(\Grid\times\GridT)$ of the scheme \eqref{FullyDiscNumSch1}.
%% \end{theorem}

\begin{theorem} \label{thm:fullydisc}
Assume  \eqref{asNumSch}, \eqref{gas}, \eqref{u_0as}, and the stencils
 $\mathcal S_1,\mathcal S_2\subset \Grid$.
\begin{enumerate}[{\rm (a)}]
\item \label{eu} \textup{(Existence/uniqueness)}   There exists a unique
  solution $U_\beta^j$ of \eqref{FullyDiscNumSch1} such that
  \begin{equation*}
  \sum_{j\in\J}\sum_{\beta}|U_\beta^j|<+\infty.
  \end{equation*}
\end{enumerate}
Let $U_\beta^j,V_\beta^j$ be solutions of the scheme
\eqref{FullyDiscNumSch1} with data $u_0,f$ and $v_0,g$ respectively. 
\begin{enumerate}[{\rm (a)}]
\addtocounter{enumi}{1}
\item \textup{(Monotonicity)}   If
  $U^0_\beta\leq V^0_\beta$ and $F^j_\beta\leq G^j_\beta$, then
  $U_\beta^{j}\leq V_\beta^{j}$.
\item  \textup{($L^1$-stability)}  
  $\displaystyle \sum_\beta |U_\beta^j|\leq \sum_\beta |U_\beta^0|+
  \sum_{l=1}^j\sum_\beta |F_\beta^l| \Delta t_{l}$.
\item  \textup{($L^\infty$-stability)}  
  $\displaystyle \sup_\beta |U_\beta^j|\leq \sup_\beta |U_\beta^0|+
  \sup_\beta \sum_{l=1}^j|F_\beta^l| \Delta t_{l}$.
\item\textup{(Conservativity)}  If $\varphi_1^h$  satisfy \eqref{philipas}, 
$\displaystyle \sum_\beta U_\beta^j=\sum_\beta U_\beta^0+ \sum_{l=1}^j
  \sum_\beta F_\beta^l\Delta t_{l}.$
\item \textup{($L^1$-contraction)}  \label{thm:fullydisc:contractive}
  $\displaystyle 
\sum_\beta (U_\beta^j-V_\beta^j)^+
\leq \sum_\beta (U_\beta^0-V_\beta^0)^+ + \sum_{l=1}^j\sum_\beta
(F_\beta^l-{G}^l_\beta)^+ \Delta t_{l}.$
\item \textup{(Equicontinuity in time)}
   \label{thm:fullydisc:eqconttime}
  If
  \eqref{AsUnifLe} holds, then for all compact sets $K\subset \R^N$, 
\begin{equation*}
\begin{split}
  &h^N\sum_{x_\beta \in \Grid\cap K} |U_\beta^j-U_\beta^{j-k}|\leq
  \Lambda_K(t_j-t_{j-k})+|K|\int_{t_{j-k}}^{t_j}\|f(\cdot,\tau)\|_{L^\infty(\R^N)}\dd
  \tau.
\end{split}
\end{equation*}
 
\end{enumerate}
\vspace{-0.2cm}
Assume in addition that  $\Delta
t=o_h(1)$, and for all $h>0$, let
$U_{\beta}^j$ be the solution of a consistent scheme
\eqref{FullyDiscNumSch1}   satisfying \eqref{asNumSch} and \eqref{AsUnifLe}.
\begin{enumerate}[{\rm (a)}]
  \addtocounter{enumi}{7}
\item \label{thm:fullydisc:conv} \textup{(Convergence)} There
exists a unique distributional solution $u\in L^1(Q_T)\cap
L^\infty(Q_T)\cap C([0,T];L_\textup{loc}^1(\R^N))$ of \eqref{E}
such that for all compact sets $K\subset \R^N$, 
\begin{equation*}
\vertiii{U-u}_{K}:=\max_{t_j\in\GridT}\left\{\sum_{x_\beta \in \Grid\cap K}\int_{x_\beta+R_h} |U_{\beta}^j-u(x,t_j)|\dd x \right\} \to 0 \quad  \text{as} \quad h\to0^+.
%\vertiii{U-u}_{K}:=\max_{t_j\in\GridT}\left\{\sum_{x_\beta \in \Grid\cap K} h^N |U_{\beta}^j-u(x_\beta,t_j)|\right\} \to 0 \quad  \text{as} \quad h\to0^+.
\end{equation*}
\end{enumerate}
\end{theorem}

\begin{remark}
Parts (a)--(g) can be formulated in terms of the space interpolant $\ol{U^j}$, e.g. the $L^1$-contraction in part \eqref{thm:fullydisc:contractive} then becomes
$$\int_{\R^N} (\ol{U^j}-\ol{V^j})^+ \dd x
\leq  \int_{\R^N}(u_{0}-v_{0})^+\dd x+\int_0^{t_j}\int_{\R^N}(f-g)^+\dd x\dd \tau.$$
%Moreover, for ``nice enough'' $u$, the convergence in \eqref{thm:fullydisc:conv} can be stated as
%\begin{equation*}
%\|\widetilde{\ol U}-u\|_{C([0,T]; L^1(K))}=
%\max_{t_j\in\GridT}\left\{\sum_{x_\beta \in \Grid\cap K} h^N |U_{\beta}^j-u(x_\beta,t_j)|\right\} \to 0 \quad  \text{as} \quad h\to0^+.
%\end{equation*}
Moreover, convergence in \eqref{thm:fullydisc:conv} can be stated in
terms of space-time interpolants as
$$
\widetilde{\ol U}\to u\qquad \text{in}\qquad C([0,T];L^1_{\mathrm{loc}}(\R^N)).
$$ 
%since $\vertiii{V}_K=\|\widetilde{\ol V}\|_{C([0,T]; L^1(K))}$ for any function $V$ on $\Grid\times\GridT$.
\end{remark}

The proofs of the above results can be found in Section \ref{proof:fullyDiscreteScheme}.

%%%%%%%%%%%%%%%%%%%%%%%%%%%%%%%%%%%%%%%%%%%%%%%%%%%%%%%%%%%%%%

\subsection{Well-posedness for bounded distributional solutions}

Theorem \ref{thm:conv} implies the existence of bounded
distributional solutions solutions of \eqref{E}, and uniqueness has
been proved in \cite{DTEnJa17b}:

\begin{theorem}[Existence and uniqueness]\label{thm:exundistsoln}
Assume \eqref{phias}, \eqref{gas}, \eqref{u_0as}, and
\eqref{muas}. Then there exists a unique distributional solution $u$
of \eqref{E} such that 
$$u\in L^1(Q_T)\cap L^\infty(Q_T)\cap C([0,T];L_\textup{loc}^1(\R^N)).$$
\end{theorem}

Another consequence of Theorem \ref{thm:conv} is that most of the a priori
results in Theorems \ref{thm:propertiesscheme},
\ref{thm:propertiesscheme:contractive}, \ref{thm:timereg} will be
inherited by the solution $u$ of \eqref{E}.

\begin{proposition}[A priori estimates]\label{propconstructeddistsol}
Assume \eqref{phias} and \eqref{muas}.
Let $u,v$ 
be the distributional solutions of \eqref{E} corresponding to $u_0,v_0$ and  $f,g$ satisfying \eqref{u_0as} and \eqref{gas} respectively. Then, for every $t\in[0,T]$:
\begin{enumerate}[{\rm (a)}]
\item \textup{(\textup{Comparison})}\label{propconstructeddistsol:comp} If $u_{0}(x)\leq v_0(x)$ and $f(x,t)\leq g(x,t)$, then $u(x,t)\leq v(x,t)$.
  \smallskip
\item \textup{(\textup{$L^1$-bound})}\label{propconstructeddistsol:intbound} 
$\|u(\cdot,t)\|_{L^1(\R^N)}\leq\|u_0\|_{L^1(\R^N)}+\int_0^t\|f(\cdot,\tau)\|_{L^1(\R^N)}\dd
  \tau$.
  \smallskip
\item
  \textup{(\textup{$L^\infty$-bound})}\label{propconstructeddistsol:bndbound}
  $\|u(\cdot,t)\|_{L^\infty(\R^N)}\leq
  \|u_0\|_{L^\infty(\R^N)}+\int_0^t\|f(\cdot,\tau)\|_{L^\infty(\R^N)}\dd
  \tau$.
\smallskip  
\item \textup{(\textup{$L^1$-contraction})}\label{propconstructeddistsol:contraction}
$$
\int_{\R^N}(u-v)^+(x,t)\dd x\leq \int_{\R^N}(u_{0}-v_{0})^+(x)\dd x+\int_0^t\int_{\R^N}(f-g)^+(x,\tau)\dd x\dd \tau.
$$
\item \textup{(\textup{Time regularity})}\label{propconstructeddistsol:timereg} For every $t, s\in[0,T]$ and every compact set $K\subset \R^N$, 
\begin{equation*}
\begin{split}
\|u(\cdot,t)-u(\cdot,s)\|_{L^1(K)}&\leq \Lambda_K(|t-s|)+|K|\int_s^t\|f(\cdot,\tau)\|_{L^\infty(\R^N)}\dd \tau.
\end{split}
\end{equation*}
\end{enumerate}
\end{proposition}

See Section \ref{sec:proofAprioriDistSol} for the proofs. Note that since we do
not have full $L^1$-convergence of approximate solutions, we cannot conclude that we inherit mass conservation from Theorem
\ref{thm:propertiesscheme}
\eqref{thm:propertiesscheme:conservative}. The result is still true and 
a proof can be found in \cite{DTEnJa17a}.

\subsection{Some extensions}

\subsubsection*{More general schemes} The proofs and estimates obtained for solutions of \eqref{defNumSch} can be transferred to the more complicated scheme
$$
\begin{cases}
U_h^j(x)=U_h^{j-1}(x)+\Delta t_j\big(\sum_{k=1}^n\Levy_k^{h}[\varphi_k^h(U_h^j)](x)+\sum_{l=n+1}^m\Levy_l^{h}[\varphi_l^h(U_h^{j-1})](x)+F^j(x)\big)\\
U_h^0(x)=u_0(x)
\end{cases}
$$
where $n,m\in\N$ with $n\leq m$.

\subsubsection*{More general equations} A close examination of the
proof of Theorem 
\ref{thm:exundistsoln}, reveals that even if we omit Definition
\ref{schCon} \eqref{schCon:Eq}, we can still obtain {\em existence} for
$L^1\cap 
L^\infty$-distributional solutions of   
$$
\begin{cases}
\dell_tu -\Operator_1[\varphi_1(u)]-\Operator_2[\varphi_2(u)]=f\qquad\qquad&\text{in}\qquad Q_T,\\
u(x,0)=u_0(x) \qquad\qquad&\text{on}\qquad \R^N.
\end{cases}
$$ 
In fact, we could handle any finite sum of symmetric L\'evy operators
acting on different nonlinearities. In this case most of the
properties of the numerical method would still hold, but maybe not
convergence. To also have convergence, we need suitable uniqueness
results for the corresponding equation. At the moment, known results
like e.g. \cite{DTEnJa17a,DTEnJa18a}, or easy extensions of these, cannot cover this case.

%%%%%%%%%%%%%%%%%%%%%%%%%%%%%%%%%%%%%%%%%%%%%%%%%%%%
%%%%%%%%%%%%%%%%%%%%%NEW SECTION%%%%%%%%%%%%%%%%%%%%%%%
%%%%%%%%%%%%%%%%%%%%%%%%%%%%%%%%%%%%%%%%%%%%%%%%%%%%

\section{Examples of schemes}
\label{sec:discNumExp}

In this section, we present possible discretizations of $\Operator$ which satisfy all the properties needed to ensure convergence of the numerical scheme, that is, they satisfy Definitions \ref{schCon} and \ref{LevyInTheClass}.
%Surprisingly, our class of operators $\Levymu$ is so wide that it contains a lot of its own numerical discretizations. It even contains common discretizations of local operators as well. In this section we present a series of possible discretizations of both nonlocal operators of the type $\Levymu$ with $\mu$ satisfying \eqref{muas} and also any second order elliptic operator of the form $L^\sigma$. 
We also test our numerical schemes on an interesting special case of \eqref{E}. All of these results (and many more) will be treated in detail in Section 4 in \cite{DTEnJa18b}; we merely include a short excerpt here for completeness.

The nonlocal operator $\Levy^\mu$ contains a singular and a nonsingular part. For $\psi\in C_\textup{c}^\infty(\R^N)$ and $r>0$,
\begin{equation*}
\begin{split}
\Levy^\mu[\psi](x)&=\textup{P.V.}  \int_{0<|z|\leq r}\big(\psi(x+z)-\psi(x)\big)\dd \mu(z)+\int_{|z|>r}\big(\psi(x+z)-\psi(x)\big)\dd \mu(z)\\
&=:\Levy_{r}^\mu[\psi](x)+\Levy^{\mu,r}[\psi](x).
\end{split}
\end{equation*}
In general we assume that $h\leq r=o_h(1)$ where $h$ is the discretization in space parameter. We will present discretizations for general measures $\mu$ and give the corresponding $L^1$ Local Truncation Error (LTE)   for the fractional Laplace case ($\dd\mu(z)=\frac{c_{N,\alpha}\dd z}{|z|^{N+\alpha}}$) to show the accuracy of the approximation. By $L^1$ LTE we mean here the quantity $\|\Operator[\psi]-\Levy^h[\psi]\|_{L^1(\R^N)}$.  

\subsection{Discretizations of the singular part $\Levy_{r}^\mu$}
We propose two discretizations:

\subsubsection*{Trivial discretization} Discretize $\Levy_{r}^\mu$ by
$\Levy^h\equiv 0$. This discretization has all the required
properties, and an $O(r^{2-\alpha})$ LTE in the case of the fractional Laplacian.

\subsubsection*{Adapted vanishing viscosity discretization} For general radially symmetric measures, the discretization takes the form
\begin{equation}\label{eq:dissing3}
\Levy^h[\psi](x):=\frac{1}{2N}\int_{|z|<r}|z|^2 \dd\mu(z)\sum_{i=1}^N\frac{\psi(x+e_ih)+\psi(x-e_ih)-2\psi(x)}{h^2}.
\end{equation}
It can be shown that the LTE is $O(r^2+h^2)$ for a general measure $\mu$ and $O(r^{4-\alpha}+h^2r^{2-\alpha})$ in the fractional Laplace case. We refer to \cite{DTEnJa18b} for the general form of \eqref{eq:dissing3} when the measure is not radially symmetric.

\subsection{Discretization of the nonsingular part $\Levy^{\mu,r}$} 

For fixed $r>0$ these discretizations will approximate
  zero order integro-differential  operators. For simplicity we
  restrict to the uniform-in-space grid $\Grid$ and quadrature rules
  defined from interpolation.
  Let $\{p_\beta\}_{\beta\in\Z^N}$ be an interpolation basis
for $\Grid$, i.e. $\sum_{\beta} p_{\beta}(x)\equiv1$ for all $x\in\R^N$ and
$p_\beta(z_\gamma)=1$ for $\beta=\gamma$ and $0$ for
$\beta\neq\gamma$. Define the corresponding interpolant of a function $\psi$ as
$I_h[\psi](z):=\sum_{\beta\neq0}\psi(z_\beta)p_\beta(z)$.

\subsubsection*{Midpoint Rule:} This corresponds to
$p_\beta(x)=1_{x_\beta+R_h}(x)$.  We approximate $\Levy^{\mu,r}$ by 
\begin{equation}\label{quad_formula}
\begin{split}
\Levy^h[\psi](x)&:=\int_{|z|>r}I_h[\psi(x+\cdot)-\psi(x)](z)\dd \mu(z)\\
&=\sum_{{|z_\beta|>r}} \left(\psi(x+z_\beta)-\psi(x)\right)\int_{|z|>r} p_\beta(z)\dd \mu(z).
\end{split}
\end{equation}
Here $\int_{|z|>r}
p_\beta(z)\dd \mu(z)=\mu \left(\left(z_\beta+R_h\right)\cap\{|z|>r\}\right)$.    
The discretization is convergent for general measures
$\mu$, and in the fractional Laplace case the LTE is
$O(r^{2-\alpha}+h)$.

\subsubsection*{Multilinear   interpolation:}
Take $p_\beta$ to be piecewise linear basis functions in one dimension, and define them in a tensorial way in higher dimensions. This gives a positive interpolation.   Again we approximate $\Levy^{\mu,r}$ by \eqref{quad_formula}.
The discretization converges for general measures $\mu$ and the LTE is
$O(h^{2}r^{-\alpha})$ in the fractional Laplace case. 
\smallskip

\subsubsection*{Higher order Lagrange interpolation:} Take $p^k_\beta$
to be the Lagrange polynomials of order $k$, defined in a tensorial way in higher dimensions. Even if $p^k_\beta$ may
take negative values for $k\geq2$, it 
is known that $\int_{\R^N} p_\beta^k(x)\dd x\geq0$ for $k\leq7$
(cf. Newton-Cotes quadratures rules). For measures $\mu$ which are 
absolutely continuous with respect to the Lebesgue measure $\dd z$
with density (also) called $\mu(z)$, we approximate $\Levy^{\mu,r}$ by
\begin{equation*}
\begin{split}
\Levy^h[\psi](x)&:=\int_{|z|>r}I_h[(\psi(x+\cdot)-\psi(x))\mu(\cdot)](z)\dd z\\
&=\sum_{|z_\beta|>r} \big(\psi(x+z_\beta)-\psi(x)\big)\mu(z_\beta)\int_{|z|>r} p^k_\beta(z)\dd z.
\end{split}
\end{equation*} 
By choosing $r=r(h)$ in a precise way, different orders of convergence
can be obtained. This discretization can also be combined with
\eqref{eq:dissing3} to further improve the orders of accuracy. In the
best case, the LTE is shown to be $O(h^{\frac{7}{12}(4-\alpha)})$ in
the fractional Laplace case.

\subsection{Second order discretization of the fractional Laplacian}
Let  
$
\Delta_h\psi(x)=\frac{1}{h^2}\sum_{i=1}^N\big(\psi(x+e_ih)+\psi(x-e_ih)-2\psi(x)\big)
$
and define the $\frac\alpha 2$-power of $\Delta_h$ as
\begin{align*}\label{discfractlapN}
(-\Delta_h)^{\frac{\alpha}{2}}[\psi](x):=\frac{1}{\Gamma(-\frac{\alpha}{2})}\int_0^\infty\left(\e^{t\Delta_h}\psi
  (x)-\psi(x)\right)\frac{\dd t}{t^{1+\frac{\alpha}{2}}}.
\end{align*}
In general, we have
$(-\Delta_h)^{\frac{\alpha}{2}}[\psi](x)=\sum_{\beta\not=0}(\psi(x+z_\beta)-\psi(x))K_{\beta,h}$
with $
K_{\beta,h}:=\frac{1}{h^\alpha}\frac{1}{\Gamma(-\frac{\alpha}{2})}\int_0^\infty G(\beta,t)\frac{\dd t}{t^{1+\frac{\alpha}{2}}}$
and $G(\beta,t):=\e^{-2Nt}\prod_{i=1}^N I_{|\beta_i|}\left(2t\right)$
where $I_m$ denotes the modified Bessel function of first kind and order $m\in
\N$. Here $G\geq0$ is the Green function of the discrete
Laplacian in $\R^N$, and hence the weights $K_{\beta,h}$ are
positive. We improve the convergence rates of \cite{CiRoStToVa18} from $O(h^{2-\alpha})$ to $O(h^2)$ (independently on $\alpha$) and extend their consistency result to dimensions higher
than one.

See \cite{CuDTG-GPa18,DTEnJa18b} for further numerical details and also \cite{LiRo18} for more information about the operator $(-\Delta_h)^{\frac{\alpha}{2}}$ in $\R^N$.

\subsection{Discretization of local operators} We approximate $L=\Delta$ by
\begin{equation*}%\label{AproxLoc}
\Levy^h[\psi](x):=\sum_{i=1}^N \frac{\psi(x+h e_i)+\psi(x-he_i) - 2\psi(x) }{h^2}.%+ \sum_{i=1}^P \frac{\psi(x-r\sigma_i) - \psi(x) }{r^2}.\\
\end{equation*}
The discretization is known to have $O(h^2)$ LTE.
 Note that general operators
 $L^\sigma=\textup{tr}(\sigma\sigma^TD^2\cdot)$ can always be 
reduced to $\Delta_{\R^M}$ for some $M\leq N$ after a
 change of variables. A
 direct discretization of $L^\sigma$ is given by 
$$
\Levy^{h,\eta}[\psi](x)=\sum_{i=1}^M \frac{I_h[\psi](x+\eta\sigma_i)+I_h[\psi](x-\eta\sigma_i) - 2\psi(x) }{\eta^2},
$$
where $I_h$ denotes the first order Lagrange interpolation on $\Grid$ (see
e.g. \cite{CaFa95} and \cite{DeJa13a,DeJa13b}). 
In this case the LTE is $O(\frac{h^2}{\eta^2}+\eta^2)$ or $O(h)$ with optimal
choice $\eta=\sqrt h$. See \cite{DTEnJa18b} for
further details.

\subsection{Numerical experiment}
As an illustration, we solve numerically a case where \eqref{E}
correspond to a one phase Stefan problem (see
e.g. \cite{BrChQu12}). We take
$\Operator=-(-\Delta)^{\frac{\alpha}{2}}$, $\alpha\in(0,2)$,
$\varphi(\xi)=\max\{0,\xi-0.5\}$, and
$f\equiv0$. The solution is plotted in Figure \ref{StefanSol} (below) for
$\alpha=1$ and initial data
$u_0(x)=\e^{-\frac{1}{4-x^2}}\mathbf{1}_{[-2,2]}(x)$. Note that even
for smooth initial data, the solution seems 
not to be smooth after some time. For a slightly different
Stefan type nonlinearity, we use the midpoint rule to
obtain $L^1$- and $L^\infty$-errors for different values of
$\alpha\in(0,2)$. See Figure \ref{StefanSolComp} (below). Due to
the nonsmoothness of the solutions, the convergence rates in $L^1$ are
better than in $L^\infty$. More details on one dimensional (and also on two dimensional) experiments can be found in \cite{DTEnJa18b}.

%\begin{figure}[h!]
% \hspace{-0.4cm}
% \vspace{-0.2cm}
%  \begin{minipage}[b]{0.5\textwidth}
%    \includegraphics[width=\textwidth]{Solution_Stefan_type_problem,s=1_sqr}
%  \end{minipage}
%    \begin{minipage}[b]{0.5\textwidth}
%    \includegraphics[width=\textwidth]{comparison_stefan_minpointrule20092017}
%  \end{minipage}
%  \caption{\footnotesize To the left: The
%    solution for $\varphi(\xi)=\max\{0,\xi-0.5\}$. To the right: $L^1$- and
%    $L^\infty$-errors for the Midpoint Rule.} 
%\label{StefanSol}
%\end{figure}

\begin{figure}[h!]
\includegraphics[width=\textwidth]{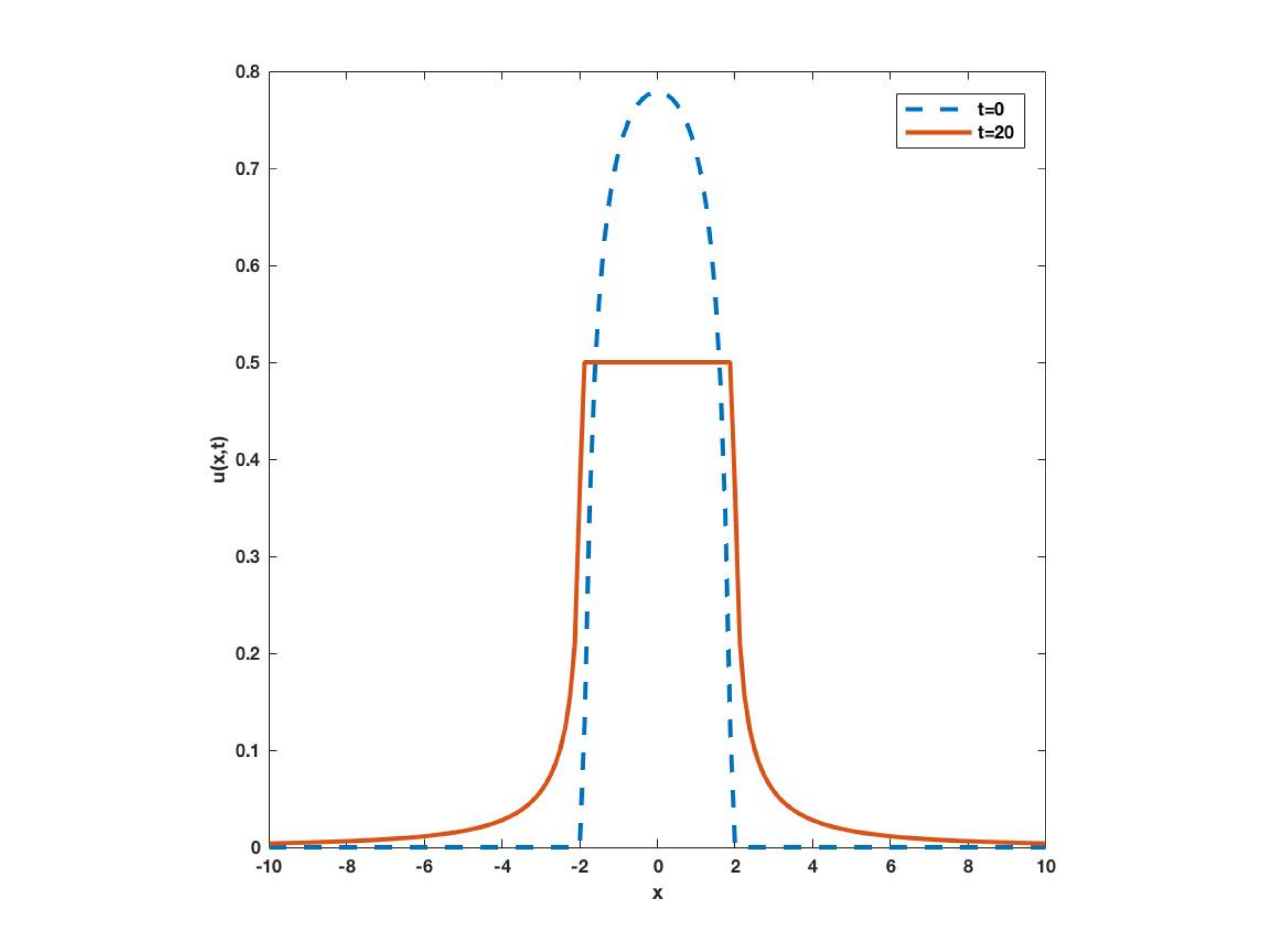}
\caption{\footnotesize The solution of a fractional Stefan problem with $\varphi(\xi)=\max\{0,\xi-0.5\}$.} 
\label{StefanSol}
\end{figure}

\begin{figure}[h!]
\includegraphics[width=\textwidth]{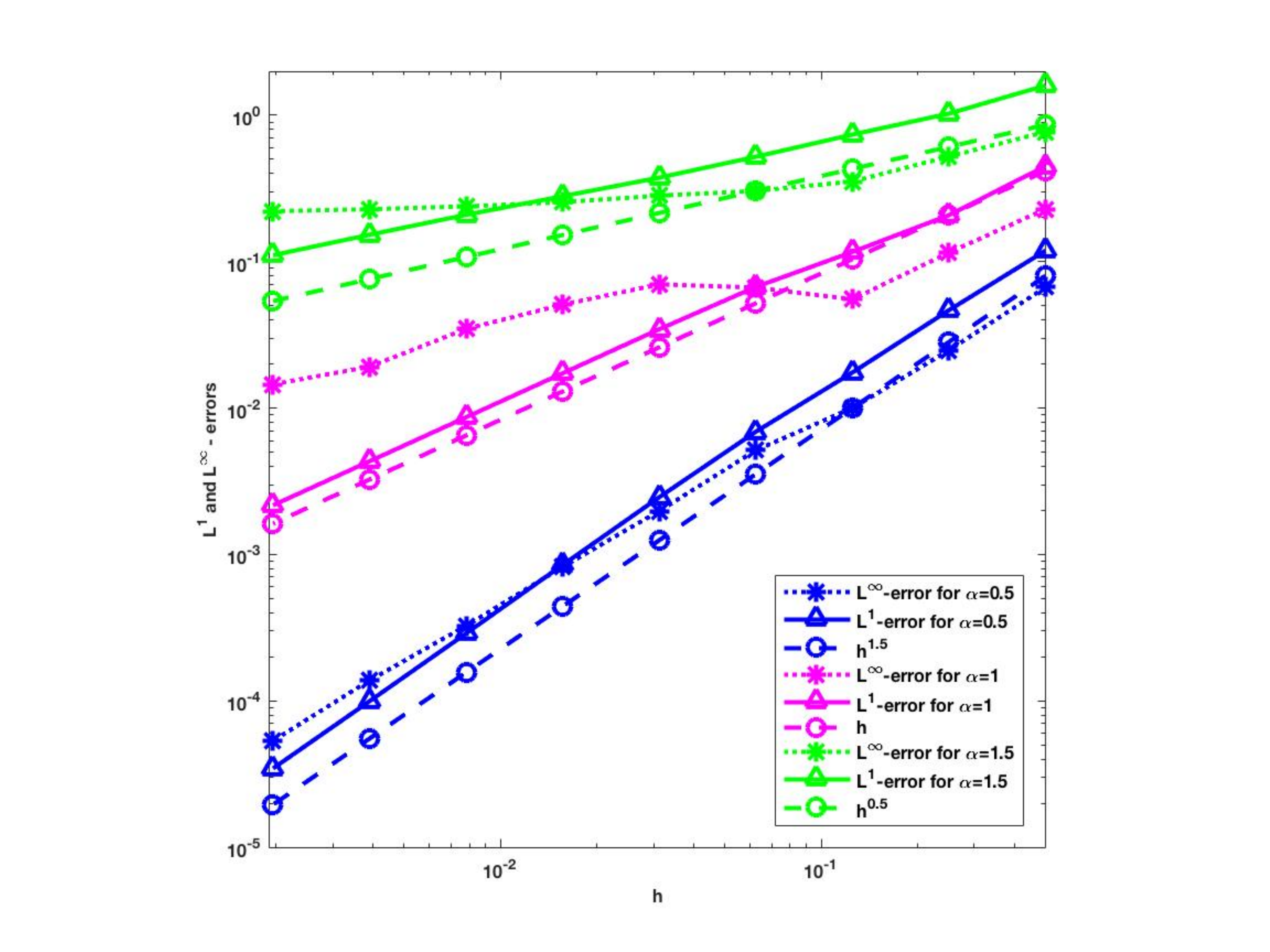}
\caption{\footnotesize The corresponding $L^1$- and $L^\infty$-errors using the Midpoint Rule.} 
\label{StefanSolComp}
\end{figure}

%%%%%%%%%%%%%%%%%%%%%%%%%%%%%%%%%%%%%%%%%%%%%%%%%%%%
%%%%%%%%%%%%%%%%%%%%%NEW SECTION%%%%%%%%%%%%%%%%%%%%%%%
%%%%%%%%%%%%%%%%%%%%%%%%%%%%%%%%%%%%%%%%%%%%%%%%%%%%

\section{Proofs of main results}

The scheme \eqref{defNumSch} can be seen as an operator splitting method
with alternating explicit and implicit steps. The explicit step is
given by the operator
\begin{equation}\label{OpT}\tag{$\Te$}
\Te[\psi](x):=\psi(x)+\Levynu[\varphi(\psi)](x) \qquad\text{for}\qquad
x\in \R^N, 
\end{equation}
while the implicit step is given by the operator
\begin{equation}\label{OpTi}\tag{$\Ti$}
\Ti[\rho](x):=w(x)\qquad \text{for}\qquad \R^N,
\end{equation} 
where $w$ is the solution of the nonlinear elliptic equation
\begin{equation}\label{EllipP}\tag{EP}
w(x)-\Levynu[\varphi(w)](x)=\rho(x)\qquad \text{in}\qquad \R^N,
\end{equation}
We can then write the scheme \eqref{defNumSch} in the following way:
\begin{equation}\label{eq:NSexpimp}
U_h^j(x)=\Ti\Big[\Te[U_h^{j-1}]+\Delta t_jF^j\Big](x),
\end{equation}
where we take  $\nu=\Delta t_j \nu_2^h$, $\varphi=\varphi_2^h$  in \eqref{OpT}
and
$\nu=\Delta t_j \nu_1^h$,  $\varphi=\varphi_1^h$
in \eqref{OpTi}.
To study the properties of the scheme \eqref{defNumSch}, we are
reduced to study the properties of the operators $\Te$ and $\Ti$.

\subsection{Properties of the numerical scheme}
\label{sec:proofPropNumSch}
In this section we prove Theorems \ref{thm:propertiesscheme} and \ref{thm:propertiesscheme:contractive}.
  We start by analyzing the operators $\Te$ and $\Ti$.
By Fubini's theorem and simple computations, we have the following
result.
\begin{lemma}\label{Levynuwell-def}
\begin{enumerate}[{\rm (a)}]
\item If \eqref{nuas} holds, $p\in\{1,\infty\}$, and $\psi\in L^p(\R^N)$,
  then $\Levynu[\psi]$ is well-defined in $L^p(\R^N)$ and
\[
\|\Levynu[\psi]\|_{L^p(\R^N)}\leq 2\|\psi\|_{L^p(\R^N)}\nu(\R^N).
\]
\item If \eqref{nuas} holds and $\psi\in L^1(\R^N)$, then $\int_{\R^N}\Levynu[\psi]\dd x=0$.
\end{enumerate}
\end{lemma}
Hence if \eqref{nuas} and \eqref{phias} hold, then $\Te$ is a
well-defined operator on $L^\infty(\R^N)$, and if $\varphi(\psi)\in L^1(\R^N)$,
then $\int \Te[\psi]\dd x=\int \psi \dd x$. For the operator $\Ti$ we
have the following result:
\begin{theorem}\label{maincor}
Assume \eqref{nuas} and \eqref{phias}. If $\rho\in L^1(\R^N)\cap
L^\infty(\R^N)$, then there exists a unique a.e.-solution $\Ti[\rho]=w\in
L^1(\R^N)\cap L^\infty(\R^N)$ of \eqref{EllipP}. 
\end{theorem}

We now list the remaining properties of $\Te$ and $\Ti$ that we
use in this section.

\begin{theorem}\label{propOpT:proplimsolnofEllipP}
 Assume   \eqref{nuas}, $\phi,\hat{\phi}\in L^1(\R^N)\cap
L^\infty(\R^N)$, and either
$$
\text{\eqref{philipas}\ \ and\ \ $\Lvar
  \nu(\R^N)\leq1$}\ \ \text{for}\ \
\Te\quad\qquad\text{or}\quad\qquad\text{\eqref{phias}}\ \ \text{for}\ \ \Ti,
$$
where
$\Lvar:=\sup_{|\zeta|\leq\max\{\|\phi\|_{L^\infty},\|\hat\phi\|_{L^\infty}\}}|\varphi'(\zeta)|$.

Whether $T=\Te$ or $T=\Ti$, it then follows that
\begin{enumerate}[{\rm (a)}]
\item \textup{(Comparison)} if $\phi\leq \hat{\phi}$ a.e., then $T[\phi]\leq
  T[\hat{\phi}]$ a.e.;
  \smallskip
\item  \textup{($L^1$-contraction)}
$\int_{\R^N}(T[\phi](x)-T[\hat{\phi}](x))^+\dd
  x\leq\int_{\R^N}(\phi(x)-\hat{\phi}(x))^+\dd x$;
  \smallskip
\item \textup{($L^1$-bound)} $\|T[\phi]\|_{L^1(\R^N)}\leq
  \|\phi\|_{L^1(\R^N)}$; and
  \smallskip
\item \textup{($L^\infty$-bound)} $\|T[\phi]\|_{L^\infty(\R^N)}\leq \|\phi\|_{L^\infty(\R^N)}$.
\end{enumerate}
\end{theorem}

The proofs of Theorems \ref{maincor} and
\ref{propOpT:proplimsolnofEllipP} will be given in Section \ref{sec:auxresults}. 

\begin{remark}
Note that $\Lvar
\nu(\R^N)\leq1$ is a CFL-condition yielding monotonicity/comparison
for the scheme.
\end{remark}

We are now ready to prove a priori, $L^1$-contraction, existence, and
uniqueness results for the numerical scheme \eqref{defNumSch}.

\begin{proof}[Proof of Theorem \ref{thm:propertiesscheme}]

\noindent\eqref{thm:propertiesscheme:monotone} Note that $U^0_h\leq
V^0_h$ and $F^j\leq G^j$. 
If
$U^{j-1}_h\leq V^{j-1}_h,$ then by Theorem
\ref{propOpT:proplimsolnofEllipP} (a),
\begin{align*}
  &(\Te[U_h^{j-1}]+\Delta t_jF^j) - (\Te[V_h^{j-1}]+\Delta t_jG^j)\\
  &=(\Te[U_h^{j-1}]-\Te[V_h^{j-1}]) + \Delta t_j(F^j-G^j)\leq0
\end{align*}
and thus, by \eqref{eq:NSexpimp} and Theorem
\ref{propOpT:proplimsolnofEllipP} (a) again,
\begin{equation*}
U^j_h-V^j_h= \Ti\Big[\Te[U_h^{j-1}]+\Delta t_jF^j\Big] -  \Ti\Big[\Te[V_h^{j-1}]+\Delta t_jG^j\Big]\leq0.
\end{equation*}
Since $U^0_h- V^0_h\leq0$, part (a) follows by induction.

\medskip
\noindent\eqref{thm:propertiesscheme:intstable}--\eqref{thm:propertiesscheme:bndstable} 
Let $X$ be either $L^1(\R^d)$ or $L^\infty(\R^d)$. By Theorem \ref{propOpT:proplimsolnofEllipP} (c) or (d),
\begin{equation*}
\begin{split}
\|U_h^j\|_{X}&=\Big\|\Ti\Big[\Te[U_h^{j-1}]+\Delta t_jF^j\Big]\Big\|_{X}\leq \|\Te[U_h^{j-1}]+\Delta t_jF^j]\|_{X}\\
&\leq\|U_h^{j-1}\|_{X}+\Delta t_j\|F^j\|_{X}.
\end{split}
\end{equation*}
Then we iterate $j$ down to zero to get
$
\|U_h^j\|_{X}\leq  \|U_h^{0}\|_{X}+\sum_{l=1}^{j}\|F^l\|_{X}\Delta t_l
$, and by the definition of $F^l$, 
\begin{equation*}
\begin{split}
&\sum_{l=1}^{j}\|F^l\|_{X}\Delta t_l=\sum_{l=1}^{j}\bigg\|\frac{1}{\Delta t_l}\int_{t_{l-1}}^{t_l}f(x,\tau)\dd \tau\bigg\|_{X}\Delta t_l
\leq\int_{0}^{t_j}\|f(\cdot,\tau)\|_{X}\dd \tau.
\end{split}
\end{equation*}

\medskip
\noindent\eqref{thm:propertiesscheme:conservative} Since
$\varphi^h_i$ is locally Lipschitz, now
$\varphi^h_i(U_h^j),\varphi^h_i(V_h^j)\in L^1$. The result then follows
from integrating \eqref{defNumSch} in $x$, iterating $j$ down to zero,
and using that the integral of nonsingular L\'evy operators acting on
integrable functions is zero (Lemma \ref{Levynuwell-def} (b)). This completes the proof.
\end{proof}

\begin{proof}[Proof of Theorem \ref{thm:propertiesscheme:contractive}]
By two applications of Theorem \ref{propOpT:proplimsolnofEllipP} (b),
\begin{equation*}
\begin{split}
&\int_{\R^N}(U_h^j-V_h^j)^+(x)\dd x\leq \int_{\R^N}(U_h^{j-1}-V_h^{j-1})^+(x)\dd x+\Delta t_j\int_{\R^N}(F^j-G^j)^+(x)\dd x.
\end{split}
\end{equation*}
Then we iterate $j$ down to zero to get
\begin{equation*}
\begin{split}
&\int_{\R^N}(U_h^j-V_h^j)^+(x)\dd x
\leq \int_{\R^N}(U_h^{0}-V_h^{0})^+(x)\dd x+\sum_{l=1}^{j}\Delta
t_l\int_{\R^N}(F^{l}-G^{l})^+(x)\dd x.
\end{split}
\end{equation*}
By the definition of $F^l$ and $G^l$, Jensen's inequality, and Tonelli's theorem,
\begin{equation*}
\begin{split}
\sum_{l=1}^{j}\Delta t_{l}\int_{\R^N}(F^{l}-G^{l})^+(x)\dd x&=\sum_{l=1}^{j}\Delta t_l\int_{\R^N}\bigg(\frac{1}{\Delta t_l}\int_{t_{l-1}}^{t_l}\big(f-g\big)(x,\tau)\dd \tau\bigg)^+\dd x\\
&\leq\int_0^{t_j}\int_{\R^N}\big(f(x,s)-g(x,s)\big)^+\dd x\dd s.
\end{split}
\end{equation*}
The proof is complete.
\end{proof}

We finish by proving existence of a unique solution of the numerical scheme.
\begin{proof}[Proof of Theorem \ref{thm:existuniqueNumSch}]
Proof by induction. Assume solutions $U_h^{i}\in L^1\cap L^\infty$ of \eqref{defNumSch} exists for
$i=1,\dots,j-1$. Then since $\rho= \Te[U_h^{j-1}]+\Delta t_jF^j\in L^1\cap L^\infty$ by Theorem \ref{propOpT:proplimsolnofEllipP} and \eqref{gas}, existence and uniqueness of an a.e.-solution $\Ti[\rho]= U_h^j\in L^1\cap L^\infty$ of \eqref{EllipP} follows by Theorem \ref{maincor}. 
In view of \eqref{eq:NSexpimp}, this $U_h^j$ is the unique a.e.-solution of
\eqref{defNumSch} at $t=t_j$. 
\end{proof}

The strategy for the remaining proofs is the following. We first prove
equiboundedness and equicontinuity results for the sequence of
interpolated solutions $\{\Ut_{h}\}_{h>0}$ of the scheme
\eqref{defNumSch} as $h\to0^+$. By Arzel\`a-Ascoli and Kolmogorov-Riesz type
compactness results, see e.g. the Appendix of \cite{DrEy16},
%\color{magenta}see Theorem A.8 in \cite{HoRi02}, \normalcolor 
we conclude that there is a convergent subsequence in $C([0,T],
L^1_{\mathrm{loc}}(\R^N))$. We use consistency to prove that any such limit
must be the unique solution of \eqref{E}. Finally, by a standard
argument combining compactness and uniqueness of limit points, we
conclude that the full sequence must converge.

\subsection{Equicontinuity and compactness of the numerical scheme}

In this section we prove Theorem \ref{thm:timereg}, equicontinuity in
space, and compactness for the scheme.
Since $\Ut_{h}$ is the interpolation of $U_h$
defined in \eqref{xt-interp}, we will prove the equiboundedness and
-continuity first for $U_h^j$ and then transfer these results to $\Ut_{h}$.  
%\begin{lemma}[Equibound \normalcolor]\label{equiboundNumSch}
%Assume \eqref{gas}, \eqref{u_0as}, \eqref{asNumSch} hold for all
%$h>0$, and let $\{U_h^j\}_{h>0}$ be a.e.-solutions of
%\eqref{defNumSch}. Then, for all $j\in\J$, 
%$$
%\sup_{h>0}\|U_h^j\|_{L^\infty(\R^N)}\leq M_{u_0,g},
%$$
%where $ M_{u_0,g}<\infty$ is defined below \eqref{timeRegConst}.
%\end{lemma}
The equiboundedness is a direct corollary of Theorem \ref{thm:propertiesscheme} \eqref{thm:propertiesscheme:bndstable}.
\begin{lemma}[Equicontinuity in space]\label{spaceRegNumSch}
Assume \eqref{u_0as}, \eqref{gas}, and \eqref{asNumSch} hold for
all $h>0$, and let $\{U_h^j\}_{h>0}$ be a.e.-solutions of
\eqref{defNumSch}. Then, for all $j\in \J$, all compact sets
$K\subset \R^N$,  and all $\eta>0$,
\begin{equation*}
\begin{split}
&\sup_{|\xi|\leq\eta}\|U_h^j-U_h^j(\cdot+\xi)\|_{L^1(K)}\leq \lambda_{u_0,f}(\eta),
\end{split}
\end{equation*}
where
$\lambda_{u_0,f}$ is defined in \eqref{omegau0f}.
\end{lemma}

\begin{proof}
 By translation invariance and uniqueness, $U_h^j(x+\xi)$ is a solution
 of \eqref{defNumSch} with data $u_0(\cdot+\xi)$ and
 $f(\cdot+\xi,\cdot)$. Taking $V_h^j(x)=U_h^j(x+\xi)$ in the
 $L^1$-contraction Theorem \ref{thm:propertiesscheme:contractive} then
 concludes the estimate. Continuity of the $L^1$-translation and
 assumptions \eqref{u_0as} and \eqref{gas} shows that $\lim_{\eta\to0}\lambda_{u_0,f}(\eta)=0$.
\end{proof}

Under the additional assumption of having a consistent numerical scheme, 
\begin{equation}\label{condTimeRegVarphiBounded}
\sup_{h\in(0,1)}\|\varphi_i^h(U_h^j)\|_{L^\infty(\R^N)}<\infty
\end{equation}
and
$$\sup_{h\in(0,1)}\|\Levy_i^{h}[\psi]\|_{L^1(\R^N)}<\infty\qquad\textup{for all}\qquad \psi\in C_\textup{c}^\infty(\R^N)
$$
for $i=1,2$. The first bound is trivial, while the second follows since
$
\|\Levy_i^{h}[\psi]\|_{L^1(\R^N)}\leq
\|\Levy_i^{h}[\psi]-\Operator_i[\psi]\|_{L^1(\R^N)}+\|\Operator_i[\psi]\|_{L^1(\R^N)}
$
is bounded for $h\leq1$ by Definition \ref{schCon} \eqref{schCon:Op}.
These facts allow us to prove the time equicontinuity result Theorem
\ref{thm:timereg}.

\begin{proof}[Proof of Theorem \ref{thm:timereg}]
We exploit the idea, which is sometimes referred to as the Kru\v{z}kov interpolation lemma \cite{Kru69}, that an estimate on the $L^1$-translations in space will give an estimate on the $L^1$-translations in time. To simplify, we start by considering right-hand sides $f=0$.

The numerical scheme \eqref{defNumSch} can be written as
\[
U_h^j(x)-U_h^{j-1}(x)=\Delta t_j\left(\Levy^h_1[\varphi_1^h(U_h^j)](x)+\Levy^h_2[\varphi_2^h(U_h^{j-1})](x)\right).
\]
Let $\omega_\delta$ be a standard mollifier in $\R^N$ obtained by scaling from a fixed $\omega$, and define
$(U_h^{j})_\delta(x):=(U_h^j*\omega_\delta)(x)$. Taking the
convolution of the scheme with $\omega_\delta$ and using the fact that
the operator $\Levynu$ commutes with convolutions, we find that
\begin{equation*}
\begin{split}
(U_h^{j})_\delta(x)-(U_h^{j-1})_\delta(x)&=\Delta t_j\left(\Levy^h_1[\varphi_1^h(U_h^j)]+\Levy^h_2[\varphi_2^h(U_h^{j-1})]\right)\ast\omega_\delta(x)\\
&=\Delta t_j\left(\varphi_1^h(U_h^j)\ast\Levy^h_1[\omega_\delta](x)+\varphi_2^h(U_h^{j-1})\ast\Levy^h_2[\omega_\delta](x)\right).
\end{split}
\end{equation*}
We integrate over any compact set $K\subset\R^N$,  use Theorem \ref{thm:propertiesscheme}
 \eqref{thm:propertiesscheme:bndstable}, and
 \eqref{condTimeRegVarphiBounded}, and standard properties of
 mollifiers (see e.g. the proof of Lemma 4.3 in
\cite{DTEnJa17a}), to get
\begin{equation}\label{ApproxTimeRegularity}
\begin{split}
&\int_{K}\big|(U_h^{j})_\delta-(U_h^{j-1})_\delta\big|\dd x\\
&\leq\Delta t_j|K|\Big(\| \varphi_1^h(U_h^j)\|_{L^\infty}\|\Levy^h_1[\omega_\delta]\|_{L^1}+ \|\varphi_2^h(U_h^{j-1})\|_{L^\infty}  \|\Levy^h_2[\omega_\delta]\|_{L^1}\Big)\\
&\leq\Delta t_j|K|\Bigg(\sup_{|r|\leq M_{u_0,f}}\!\!|\varphi_1^h(r)|\,\|\Levy^h_1[\omega_\delta]\|_{L^1}+\sup_{|r|\leq M_{u_0,f}}\!\!|\varphi_2^h(r)|\, \|\Levy^h_2[\omega_\delta]\|_{L^1}\Bigg)\\
&\leq C_K\Delta t_j (1+\delta^{-2}),
\end{split}
\end{equation}
where $C_K=C_{K,u_0,f,\varphi_1,\varphi_2,\nu_1,\nu_2}$ is given by
\eqref{timeRegConst} with constant $c$ such that $c(1+\delta^{-2})$ is
 a uniform in $h$  upper bound on 
$\max_{i=1,2}\|\Levy^h_i[\omega_\delta]\|_{L^1}$. This upper bound follows
from \eqref{eq:WPL}, the uniform L\'evy condition \eqref{AsUnifLe}, and the properties of $\omega_\delta$:
\begin{equation*}
\begin{split}
\|\Levy^h_i[\omega_\delta]\|_{L^1}&\leq C\|D^2\omega_\delta\|_{L^1}
  \Big(\int_{|z|\leq1}|z|^2\dd\nu^h_i(z)\Big)+ 2\|\omega_\delta\|_{L^1}  \int_{|z|>1}\dd\nu^h_i(z)\\
  &\leq \delta^{-2} C\|D^2\omega\|_{L^1}
  \Big(\int_{|z|\leq1}|z|^2\dd\nu^h_i(z)\Big) +2\|\omega\|_{L^1}
  \int_{|z|>1}\dd\nu^h_i(z).
  %% \\
  %% &\leq C \left(\|D^2\omega\|_{L^1} + \|\omega\|_{L^1} \right) \sup_{h<1} \left(\int_{|z|>0} |z|^2\wedge 1 \dd \nu_i^h\right) (1+\delta^{-2}).
  \end{split}
\end{equation*}
By iterating \eqref{ApproxTimeRegularity} and using
Tonelli plus Theorem \ref{thm:propertiesscheme:contractive}, we obtain
\begin{equation*}
\begin{split}
\|U_h^{j}-U_h^{j-k}\|_{L^1(K)}&\leq  \|U_h^{j}-(U_h^{j})_\delta\|_{L^1(K)} + \|(U_h^{j})_\delta-(U_h^{j-k})_\delta\|_{L^1(K)}    \\
&\quad+   \|(U_h^{j-k})_\delta-U_h^{j-k}\|_{L^1(K)}\\
&\leq  2\sup_{|h|\leq\delta}\|u_0-u_0(\cdot+h)\|_{L^1(\R^N)}+C_K(t_j-t_{j-k})(1+\delta^{-2}),
\end{split}
\end{equation*}
Now we conclude by taking $\delta=(t_{j}-t_{j-k})^{\frac{1}{3}}$.

The proof for $f\not\equiv0$ follows in a similar way after using that 
\begin{align}\label{eq:IneqF}
\sum_{l=j-k+1}^{j}\|F^l\|_{L^\infty(\R^N)}\Delta t_l \leq\int_{t_{j-k}}^{t_j}\|f(\cdot,\tau)\|_{L^\infty(\R^N)}\dd
    \tau
\end{align}
for $F^l$ defined as in \eqref{sourceNumTime}.
\end{proof}

The equiboundedness, Lemma
\ref{spaceRegNumSch}, and Theorem \ref{thm:timereg} (plus  Theorems
\ref{thm:propertiesscheme} and \ref{thm:propertiesscheme:contractive})
immediately transfers, mutatis mutandis, to $\Ut_h$. We only restate
the (slightly modified) equicontinuity in time result for $\Ut_h$
here: 

\begin{lemma}[Equicontinuity in time]\label{lem:timereginterpolant}
Assume \eqref{gas}, \eqref{u_0as}, $\Delta t=o_h(1)$, and for all $h>0$, let $U^j_h$ be the solution of a consistent scheme \eqref{defNumSch} satisfying \eqref{asNumSch} and \eqref{AsUnifLe}. Then, for $t,s\in[0,T]$,
\begin{equation}\label{res:timereginterpolant}
\begin{split}
\|\Ut_h(\cdot,t)-\Ut_h(\cdot,s)\|_{L^1(K)} &\leq \Lambda_K(|t-s|)\\
&\quad+|K|\Big(\int_{s}^{t}\|f(\cdot,\tau)\|_{L^\infty(\R^N)}\dd\tau+\lambda(|t-s|,\Delta t)\Big),
\end{split}
\end{equation}
where $\lambda$ is continuous and satisfies
\begin{equation}\label{otdt}
\sup_{\Delta t\leq1}|\lambda(\delta,\Delta
t)|\stackrel{\delta\to0}{\longrightarrow}0,\quad\text{and for $\delta\in[0,T]$,}\quad\lambda(\delta,\Delta
  t)\stackrel{\Delta t\to0}{\longrightarrow}0.
\end{equation}
\end{lemma}

\begin{proof}
\noindent\textbf{1)} \textit{Assume $f=0$.} The proof of \eqref{res:timereginterpolant} is like the proof of Theorem \ref{thm:timereg} with a slightly modified end where the time interpolant \eqref{xt-interp} appears:
For $t\in(t_{j-1},t_j]$ and $s\in(t_{j-k-1},t_{j-k}]$, 
\begin{equation*}
\begin{split}
\|(\Ut_{h})_\delta(\cdot,t)-(\Ut_{h})_\delta(\cdot,s)\|_{L^{1}}
&\leq \big\|(\Ut_{h})_\delta(\cdot,t)-(U_{h}^{j-1})_\delta\|_{L^1}+\|(U_h^{j-k})_\delta-(\Ut_{h})_\delta(\cdot,s)\|_{L^1}.\\
&\quad+\sum_{l=j-k+1}^{j-1}\|(U^{l}_h)_\delta-(U^{l-1}_h)_\delta\|_{L^{1}}
%&\quad+\|U_h^{j-k}-\Ut_{h}(\cdot,s)\|_{L^{1}}.
\end{split}
\end{equation*}
Since by the definition of linear interpolation,
\begin{align*}
&\big\|(\Ut_{h})_\delta(\cdot,t)-(U^{j-1}_h)_\delta\|_{L^{1}}
\leq \frac{t-t_{j-1}}{\Delta
  t_j}\big\|(U^j_h)_\delta-(U^{j-1}_h)_\delta\big\|_{L^{1}},\\
&\|(U^{j-k}_h)_\delta-(\Ut_{h})_\delta(\cdot,s)\|_{L^{1}}\leq\frac{t_{j-k}-s}{\Delta t_{j-k}}\big\|(U^{j-k}_h)_\delta-(U^{j-k-1}_h)_\delta\big\|_{L^{1}},
\end{align*}
it follows by repeated use of \eqref{ApproxTimeRegularity} that
\begin{align*}
  \|(\Ut_{h})_\delta(\cdot,t)-(\Ut_{h})_\delta(\cdot,s)\|_{L^{1}}&\leq
  \Big((t-t_{j-1})+\sum_{l=j-k+1}^{j-1}\Delta t_{l}+(t_{j-k}-s)\Big)C_K
  (1+\delta^{-2})\\
  &= (t-s)C_K (1+\delta^{-2}).  
\end{align*}
At this point we can conlclude the proof as before when $f=0$.
\medskip

\noindent \textbf{2)} \textit{Assume $f\not\equiv0$.} From the proof of Theorem \ref{thm:timereg} with $f\not\equiv0$ and an inequality as \eqref{eq:IneqF}, we find that
\begin{equation*}
\begin{split}
\lambda(|t-s|,\Delta t):=\int_{s}^t\big(\hat{F}_{\Delta t}(\tau)-\|f(\cdot,\tau)\|_{L^\infty(\R^N)}\big)\dd \tau,
\end{split}
\end{equation*}
where the piecewise constant function $\hat{F}_{\Delta t}$ is defined from $\tau\mapsto\|f(\cdot,\tau)\|_{L^\infty(\R^N)}$ by averages:
\begin{equation*}
\hat{F}_{\Delta t}(\tau):= 
\frac{1}{\Delta
  t_l}\int_{t_{l-1}}^{t_l}\|f(\cdot,\tau')\|_{L^\infty(\R^N)}\dd \tau'
\qquad\text{for}\qquad \tau\in(t_{l-1},t_l], \quad l\in \mathbb J. 
\end{equation*}
A standard argument shows that $\hat{F}_{\Delta
  t}\to\|f(\cdot,\tau)\|_{L^\infty(\R^N)}$ in 
$L^1(0,T)$ as $\Delta t\to0^+$, and then the sequence $\{\hat{F}_{\Delta
  t}-\|f(\cdot,\tau)\|_{L^\infty(\R^N)}\}_{\Delta t\leq 1}$ is
equi-integrable by the Vitalli convergence theorem. Since
equi-integrability implies that
$$
\lim_{|t-s|\to0}\sup_{\Delta t\leq 1}\int_{[s,t]}\big|\hat{F}_{\Delta t}(\tau)-\|f(\cdot,\tau)\|_{L^\infty(\R^N)}\big|\dd \tau=0,
$$
the two claims in \eqref{otdt} readily follows.
%Take a continuous $\psi(t)$, and a piecewise constant approximation
%$\hat{\psi}_{\Delta t}(t)$ defined from $\psi$ by averages as in the
%definition of $\hat F_{\Delta t}$ above, and write 
%\begin{equation*}
%\begin{split}
%\lambda(|t-s|,\Delta t)&=\int_{s}^t\big[\hat{F}_{\Delta t}(\tau)-\hat{\psi}_{\Delta t}(\tau)\big]+\big[\hat{\psi}_{\Delta t}(\tau)-\psi(\tau)\big]+\big[\psi(\tau)-\|f(\cdot,\tau)\|_{L^\infty}\big]\dd \tau.
%\end{split}
%\end{equation*} 
%The first and third terms are both bounded by $\|\psi-\|f(\cdot,\cdot)\|_{L^\infty(\R^N)}\|_{L^1(0,T)}$, while
%\begin{equation*}
%\begin{split}
%\bigg|\int_{s}^t\big(\hat{\psi}_{\Delta t}(\tau)-\psi(\tau)\big)\dd
%\tau\bigg|&=\bigg|\Big[\frac{1}{\Delta
%  t_j}\int_{t_{j-1}}^t\int_{t_{j-1}}^{t_j}
%+\frac{1}{\Delta
%  t_{j-k}}\int_{s}^{t_{j-k}}\int_{t_{j-k-1}}^{t_{j-k}}\Big]\big(\psi(\tau')-\psi(\tau)\big)\dd
%\tau'\dd \tau\bigg|
%\\
%&\leq |t-s|\max_{|\tau'-\tau|\leq \Delta t}|\psi(\tau')-\psi(\tau)|.
%\end{split}
%\end{equation*}
%Note the cancellations of all integrals on $(t_{j-k},t_{j-1})$.
%We conclude that
%$$
%\lambda(|t-s|,\Delta t)\leq 2\|\psi-\|f(\cdot,\cdot)\|_{L^\infty(\R^N)}\|_{L^1(0,T)}+ |t-s|\max_{|\tau'-\tau|\leq \Delta t}|\psi(\tau')-\psi(\tau)|.
%$$
%Now  \eqref{otdt} easily follows by taking infimum over $\psi$
%(think of approximation of $\|f(\cdot,t)\|_{L^\infty}$ in $L^1$).
\end{proof}

In view of equiboundedness and -continuity of $\{\Ut_h\}_{h>0}$, we can now use the Arzel\`a-Ascoli and Kolmogorov-Riesz type compactness results, see e.g. the Appendix of \cite{DrEy16},
%Theorem A.8 in \cite{HoRi02} \normalcolor 
to conclude the following result:

\begin{theorem}[Compactness]\label{compIntNumSch}
Assume \eqref{gas}, \eqref{u_0as}, $\Delta t=o_h(1)$,
\eqref{defNumSch} is a consistent scheme  satisfying \eqref{AsUnifLe} and  such that \eqref{asNumSch}
holds for every $h>0$, let $\{U_h^j\}_{h>0}$ be the solutions of \eqref{defNumSch}
and $\{\Ut_h\}_{h>0}$ their time interpolants defined in \eqref{xt-interp}.
Then there exists a subsequence $\{\Ut_{h_n}\}_{n\in\N}$ and a $u\in C([0,T];L_\textup{loc}^1(\R^N))$ such that
$$
\Ut_{h_n}\to u \qquad\text{in}\qquad
C([0,T];L_\textup{loc}^1(\R^N))\quad\text{and\quad a.e.}\qquad\text{as}\qquad n\to\infty.
$$
\end{theorem}

\subsection{Convergence of the numerical scheme}
\label{sec:proofConvNumSch}

In this section we prove convergence of the scheme, Theorem
\ref{thm:conv}.
We start with a consequence of the consistency and stability of the
scheme and the stability of the equation. 
\begin{lemma}\label{lem:conv}
Under the assumptions of Theorem \ref{compIntNumSch}, any subsequence
of $\{\Ut_h\}_{h>0}$ that converges in
$C([0,T];L_\textup{loc}^1(\R^N))$, converges to a distributional
solution $u\in L^1(Q_T)\cap  L^\infty(Q_T)$ of \eqref{E}. 
\end{lemma}

An immediate corollary of this lemma, the compactness in Theorem
\ref{compIntNumSch}, and uniqueness in Theorem \ref{unique}, is then
the following result. 
\begin{corollary}\label{cor}
Under the assumptions of Theorem \ref{compIntNumSch}, any subsequence
of $\{\Ut_h\}_{h>0}$ has a further subsequence that converges in
$C([0,T];L_\textup{loc}^1(\R^N))$ to the unique distributional solution $u\in L^1(Q_T)\cap 
L^\infty(Q_T)$ of \eqref{E}. 
\end{corollary}

We now prove convergence of the scheme, Theorem \ref{thm:conv}.

\begin{proof}[Proof of Theorem \ref{thm:conv}]
By compactness, Theorem \ref{compIntNumSch}, there is a subsequence of
$\{\Ut_h\}_{h>0}$ that converge to some function $u$ in
$C([0,T];L_\textup{loc}^1(\R^N))$. By Lemma \ref{lem:conv}, $u$ is a
distributional solution of \eqref{E} belonging to $L^1(Q_T)\cap
L^\infty(Q_T)$. Then the whole sequence converges since it is bounded
and other limit points are excluded by Corollary \ref{cor}. 
%Now assume by contradiction that there is a
%subsequence of $\{\Ut_h\}_{h>0}$ that does not converge to $u$ in
%$C([0,T];L_\textup{loc}^1(\R^N))$. Then there is a further subsequence
%and an $\veps>0$ such that $d(u,\Ut_{h_{k_j}})>\veps$ for every
%$j\in\N$, where $d$ is a distance in
%$C([0,T];L_\textup{loc}^1(\R^N))$. But this is not possible in view of
%Corollary \ref{cor}, and hence the whole sequence $\{\Ut_h\}_{h>0}$
%converge to $u$ in $C([0,T];L_\textup{loc}^1(\R^N))$.
\end{proof}
It remains to prove Lemma \ref{lem:conv}.

\begin{proof}[Proof of Lemma \ref{lem:conv}]\
  \smallskip
\noindent Take any  $C([0,T];L_\textup{loc}^1(\R^N))$ converging subsequence of
$\{\Ut_{h}\}_{h>0}$ and let $u$ be its limit. For simplicity we 
 also denote the subsequence by $\{\Ut_{h}\}_{h>0}$. Remember that
$\Ut_{h}$ is the time interpolation of $U_h$
defined in \eqref{xt-interp}.

\medskip
\noindent \textbf{1)}  {\it The limit $u\in L^1(Q_T)\cap L^\infty(Q_T)$.}
There is a further subsequence converging to $u$ for all $t$
and a.e. $x$. Hence we find that the $L^\infty$ bound of Theorem
\ref{thm:propertiesscheme} \eqref{thm:propertiesscheme:bndstable} is
inherited by $u$. Similarly, by Fatou's lemma, also the $L^1$
bound of Theorem \ref{thm:propertiesscheme}
\eqref{thm:propertiesscheme:intstable} carries over to $u$.
Hence we can conclude that $u\in L^1(Q_T)\cap 
L^\infty(Q_T)$.
\smallskip

\noindent We proceed to prove that $u$ is a
distributional solution of \eqref{E}, see Definition
\ref{distsol}.

\medskip
\noindent \textbf{2)} {\it Weak formulation of the numerical scheme \eqref{defNumSch}.}
Let $\psi \in C_\textup{c}^\infty(\R^N\times[0,T))$.
We multiply the scheme \eqref{defNumSch} by $\psi(x,t_{j-1})\Delta
t_j$, integrate in space, sum in time, and use the self-adjointness of
$\Levy^h_1, \Levy^h_2$, to get
\begin{equation}\label{discretedistform}
\begin{split}
 \int_{\R^N}\sum_{j=1}^{J}&\frac{U_h^{j}-U_h^{j-1}}{\Delta t_j} \psi(x,t_{j-1})\Delta t_j\dd x  = \int_{\R^N}\sum_{j=1}^{J}\varphi_1^h(U_h^{j})\Levy^h_1[\psi(\cdot,t_{j-1})]\Delta t_j\dd x\\
 \quad &+\int_{\R^N} \sum_{j=1}^{J} \varphi_2^h(U_h^{j-1})\Levy^h_2[\psi(\cdot,t_{j-1})]\Delta t_j\dd x + \int_{\R^N}\sum_{j=1}^JF^j(x)\psi(x,t_{j-1})\Delta t_j\dd x.
 \end{split}
\end{equation}
In the rest of the proof we will show that the different terms in this
equation converge to the corresponding terms in \eqref{def_eq} and thereby
conclude the proof. 

\medskip
\noindent \textbf{3)} {\it Convergence to the time derivative.}
By summation by parts, $U_h^0=u_0$, and $\psi(x,t_{J-1})=0$ for
$\Delta t$ small enough since $\psi$ has compact support,
\begin{equation*}
\begin{split}
&\int_{\R^N}\sum_{j=1}^{J}\frac{U_h^{j}(x)-U_h^{j-1}(x)}{\Delta t_j}
  \psi(x,t_{j-1})\Delta t_j\dd x\\
&=-\int_{\R^N}\sum_{j=1}^{J-1}U_h^j(x)\frac{\psi(x,t_j)-\psi(x,t_{j-1})}{\Delta
    t_j}\Delta t_j\dd x\\
  &\quad+\int_{\R^N}U_h^{J}(x)\psi(x,t_{J-1})\dd
  x-\int_{\R^N}U_h^0(x)\psi(x,0)\dd x\\
  &= - I + 0-\int_{\R^N}u_0(x)\psi(x,0)\dd x.
\end{split}
\end{equation*}

To continue, we note that for any $r>0$,
$$\sum_{j=1}^{J-1}\frac{\psi(x,t_j)-\psi(x,t_{j-1})}{\Delta
  t_j}\indik_{[t_{j-1},t_j)}(t)\to \dell_t\psi(x,t)\quad \text{in}\quad L^\infty(\R^N\times[0,T-r))$$
 as $\Delta t\to 0^+$.
Then since $U_h$ is uniformly bounded and $\Ut_h$ converges to $u$ in
$C(0,T;L_\textup{loc}^1(\R^N))$, and $\psi$ has compact support,
  a standard argument shows that
  $$I=\int_{\R^N}\sum_{j=1}^{J-1}U_h^j(x)\frac{\psi(x,t_j)-\psi(x,t_{j-1})}{\Delta
    t_j}\Delta t_j\dd x\to \int_{\R^N}\int_0^Tu(x,t)\dell_t\psi(x,t)\dd t\dd
  x$$
  as $h\to 0^+$. Combining all estimates, we conclude that as $h\to 0^+$, 
\begin{align*}
&\int_{\R^N}\sum_{j=1}^{J}\frac{U_h^{j}-U_h^{j-1}}{\Delta t_j}
\psi(x,t_{j-1})\Delta t_j\dd
x\\
&\to-\int_{\R^N}\int_0^Tu\,\dell_t\psi\dd t\dd
x-\int_{\R^N}u_0(x)\psi(x,0)\dd x.
\end{align*}

\medskip
\noindent \textbf{4)} {\it Convergence of the nonlocal terms.} 
We start by the $\Levy^h_1$-term. By adding and subtracting terms we
find that
\begin{align*}
\int_{\R^N}\sum_{j=1}^J\varphi_1^h(U_h^j)\Levy^h_1[\psi(\cdot,t_{j-1})]\Delta
  t_j\dd x=&\int_0^T\int_{\R^N}\varphi_1(u)\Operator_1[\psi(\cdot,t)]\dd t\dd
  x\\ &+E_1+E_2+E_3+E_4,
\end{align*}
where
\begin{align*}
&|E_1|\leq \int_{\R^N}\sum_{j=1}^J\int_{t_{j-1}}^{t_j}|\varphi_1^h(U_h^j)|\big|\Levy^h_1[\psi(\cdot,t_{j-1})]-\Operator_1[\psi(\cdot,t_{j-1})]\big|\dd t\dd x\\
&|E_2| \leq\int_{\R^N}\sum_{j=1}^J\int_{t_{j-1}}^{t_j}|\varphi_1^h(U_h^j)|\big|\Operator_1[\psi(\cdot,t_{j-1})]-\Operator_1[\psi(\cdot,t)]\big|\dd t\dd x\\
&|E_3| \leq  \int_{\R^N}\sum_{j=1}^J\int_{t_{j-1}}^{t_j}\big|\varphi_1^h(U_h^j)-\varphi_1(U_h^j)\big||\Operator_1[\psi(\cdot,t)]|\dd t\dd x\\
&|E_4|
 \leq\int_{\R^N}\sum_{j=1}^J\int_{t_{j-1}}^{t_j}\big|\varphi_1(U_h^j(x))-\varphi_1(u(x,t))\big||\Operator_1[\psi(\cdot,t)]|\dd
 t\dd x.
\end{align*}

First note that by \eqref{muas} and Remark \ref{addingconstants} (b),
$$\sup_{t\in[0,T]}\|\Operator_1[\psi(\cdot,t)]\|_{L^1}\leq
C\sup_{t\in[0,T]}(\|D^2\psi(\cdot,t)\|_{L^1}+\|\psi(\cdot,t)\|_{L^1})=:K<\infty.$$
Then by consistency (Definition \ref{schCon} \eqref{schCon:Op}),
% and taking $h$ small enough,
$$
\sup_{t\in[0,T]}\|(\Operator_1-\Levy^h_1)[\psi(\cdot,t)]\|_{L^1(\R^N)}\leq\sup_{t\in[0,T]}\|\psi(\cdot,t)\|_{W^{k_1,1}(\R^N)}o_h(1)\stackrel{h\to0^+}{\longrightarrow}0.
$$
%$$
%\sup_{t\in[0,T]}\|(\Operator_1-\Levy^h_1)[\psi(\cdot,t)]\|_{L^1(\R^N)}\underset{h\to0^+}{\longrightarrow}0\color{magenta}\quad
%\text{and}\quad \sup_{t\in[0,T]}\|\Levy^h_1[\psi(\cdot,t)]\|_{L^1(\R^N)}\leq 2K.\normalcolor
%$$
%Observe that here we can assume by density in $L^\infty((0,T);L^1(\R^N))$ that
%$\psi(x,t)=\Theta(t)\Gamma(x)$ and then use consistency for
%$(\Operator_1-\Levy^h_1)[\Gamma]$. \normalcolor
 By the uniform boundedness of $U_h$ (Theorem \ref{thm:propertiesscheme}
\eqref{thm:propertiesscheme:bndstable}) continuity of $\varphi$
\eqref{phias}, it first follows that $\|\varphi_1(U_h^j)\|_{L^\infty(Q_T)}\leq C$,
and then by the uniform convergence of
$\varphi_1^h\to \varphi_1$ (Definition \ref{schCon} \eqref{schCon:Non}) and taking
$h$ small enough, 
$$\|\varphi_1^h(U_h^j)-\varphi_1(U_h^j)\|_{L^\infty(Q_T)}\stackrel{h\to0^+}{\longrightarrow}0
\quad\text{and}\quad \|\varphi_1^h(U_h^j)\|_{L^\infty(Q_T)}\leq 2C.$$
From these considerations we can immediately conclude that $E_1,
E_3\to 0$ as $h\to 0^+$.

To see that $E_2\to0$, we now only need to observe that by linearity
of $\Operator$ and a Taylor expansion,
$$
\|\Operator_1[\psi(\cdot,t_{j-1})]-\Operator_1[\psi(\cdot,t)]\|_{L^1}\leq\Delta t\sup_{s\in[0,T]}\|\Operator_1[\dell_t\psi(\cdot,s)]|_{L^1},
$$
and that $\|\Operator_1[\dell_t\psi(\cdot,s)]|_{L^1}\leq C\sup_{s\in[0,T]}\big(\|D^2\dell_t\psi(\cdot,s))\|_{L^1}+\|\dell_t\psi(\cdot,s)\|_{L^1}\big)<\infty$.
Finally, we see that $E_4\to0$ by the dominated convergence theorem
since $\varphi_1(\Ut_h)$ is uniformly bounded and we may assume (by
taking a further subsequence if necessary) $\Ut_h\to u$ a.e. and
hence $\varphi_1(\Ut_h)\to \varphi_1(u)$ a.e. in
$Q_T$ as $h\to0^+$ by \eqref{phias}.

A similar argument shows the convergence of the $\Levy^h_2$-term, and
we can therefore conclude that as $h\to0^+$,
\begin{equation*}
\begin{split}
&\int_{\R^N}\sum_{j=1}^{J}\varphi_1^h(U_h^{j})\Levy^h_1[\psi(\cdot,t_{j-1})]\Delta t_j\dd x+\int_{\R^N} \sum_{j=1}^{J} \varphi_2^h(U_h^{j-1})\Levy^h_2[\psi(\cdot,t_{j-1})]\Delta t_j\dd x\\
&\to \int_{\R^N}\int_0^T\varphi_1(u)\Operator_1[\psi(\cdot,t)]\dd t\dd x+ \int_{\R^N}\int_0^T\varphi_2(u)\Operator_2[\psi(\cdot,t)]\dd t\dd x.
\end{split}
\end{equation*}

\medskip
\noindent \textbf{5)} {\it Convergence to the right-hand side.}
By the definition of $F^j$,
\begin{align*}
&\bigg|\int_{\R^N}\bigg(\sum_{j=1}^JF^j(x)\psi(x,t_{j-1})\Delta t_j-\int_0^Tf(x,t)\psi(x,t)\dd t\bigg)\dd x\bigg|\\
&\leq\int_{\R^N}\sum_{j=1}^J\int_{t_{j-1}}^{t_{j}}|\psi(x,t)-\psi(x,t_{j-1})||f(x,t)|\dd t\dd x\\
&\leq\|\dell_t\psi\|_{L^\infty(Q_T)}\|f\|_{L^1(Q_T)}\Delta t\to0^+\qquad\text{as}\qquad h\to0^+.
\end{align*}

\medskip
\noindent \textbf{6)} {\it Conclusion.} In view of steps 3) -- 5) and Definition
\ref{schCon} \eqref{schCon:Eq}, if we pass to the limit  as $h\to0^+$ in
\eqref{discretedistform}, we find that $u$ satisfy \eqref{def_eq}. In
view of step 1), $u$ is then a distributional solution of \eqref{E}
according to Definition \ref{distsol}.

\end{proof}

\subsection{Numerical schemes on uniform spatial grids}
\label{proof:fullyDiscreteScheme}

 \begin{proof}[Proof of Proposition \ref{prop:equivalence}]

\begin{figure}[h!] 
\centering
\includegraphics[width=\textwidth]{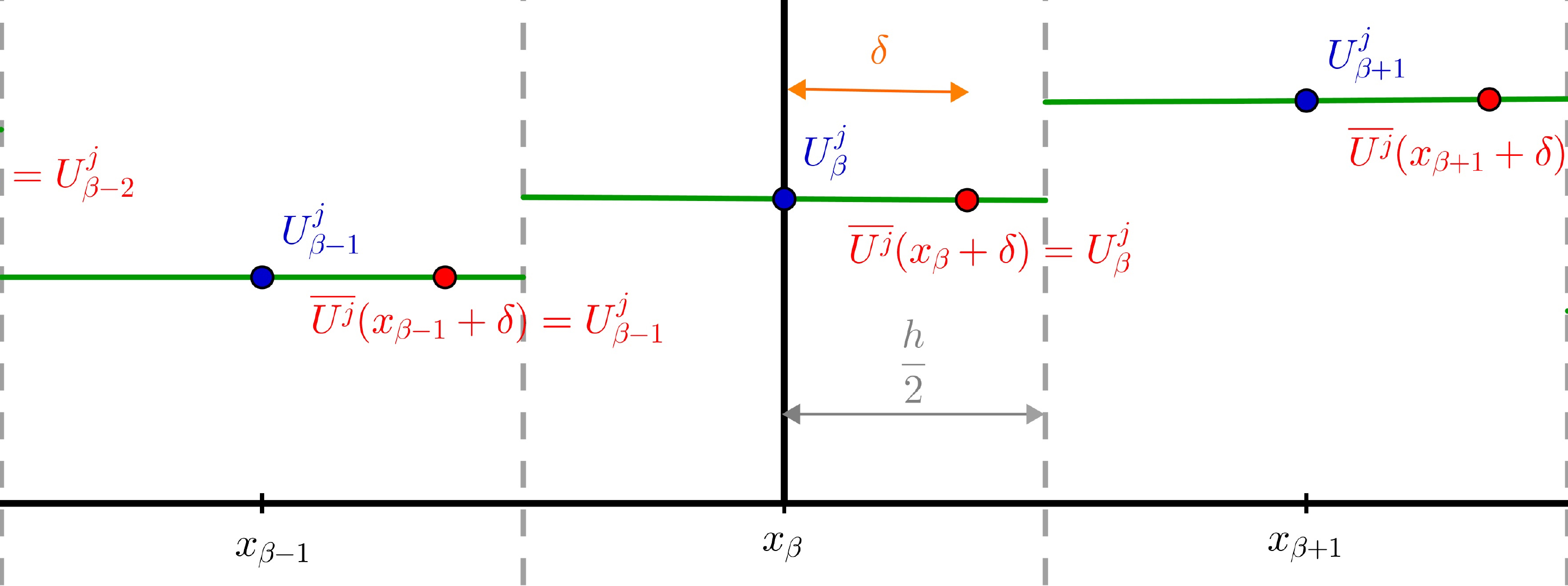}
\vspace{-0.6cm}
\caption{\footnotesize The relation between $U_\beta^j$ and
  $\overline{U^j}$ in Proposition \ref{prop:equivalence}.}
\label{fig:UnifGrid}
\end{figure}

See Figure \ref{fig:UnifGrid} for the relation between $U_\beta^j$ and $\overline{U^j}$. 

\medskip
\noindent(a) Let $x_\beta\in\Grid$.
Since the scheme \eqref{defNumSch} is translation invariant in $x$, it
follows that $U^{j}$ and $U^{j}(\cdot+y)$ are 
solutions of \eqref{defNumSch} with $\ol{U^0}, \ol{F^j}$ and
$\ol{U^0}(\cdot+y), \ol{F^j}(\cdot+y)$ as data respectively. By
uniqueness (Theorem \ref{thm:existuniqueNumSch}) and the
fact that $\ol{U^0}(x_\beta)=U_\beta^0=\ol{U^0}(x_\beta+y)$ and
$\ol{F^j}(x_\beta)=F_\beta^j=\ol{F^j}(x_\beta+y)$ for all $y\in R_h$
and $j>0$, we get that $U^j$ is constant on $x_\beta+R_h$ (a.e.) for all $j>0$.
% ************
% Let $x_\beta\in\Grid$.  For the explicit part, we have
% that if $U^{j-1}$ and $F^j$ are constant on $x_\beta+R_h$ (a.e.), then
% so is $U^{j}$ since $\Levy_i^h$ gives the same value on $x_\beta+R_h$
% (see below). The data $\ol{U^0}$ and $\ol{F^j}$ are, in fact, constant
% on $x_\beta+R_h$ (a.e.), and then by induction, so is the
% a.e.-solution $U^{j}$ of \eqref{defNumSch} for all $j>0$. For the
% implicit part, we use that the scheme \eqref{defNumSch} is translation
% invariant in $x$ to see that both $U^{j}$ and $U^{j}(\cdot+y)$ are
% solutions of \eqref{defNumSch} with $\ol{U^0}, \ol{F^j}$ and
% $\ol{U^0}(\cdot+y), \ol{F^j}(\cdot+y)$ as respective data. However, by
% using the uniqueness in Theorem \ref{thm:existuniqueNumSch} and the
% fact that $\ol{U^0}(x)=U_\beta^0=\ol{U^0}(x+y)$ and
% $\ol{F^j}(x)=F_\beta^j=\ol{F^j}(x+y)$ where $x\in x_\beta+R_h$ for
% some $\beta$ and $y$ such that $x+y\in x_\beta+R_h$, we get that $U^j$
% is constant on $x_\beta+R_h$ for all $j>0$. 
Take a
piecewise constant   version of $U^{j}$ and let
$U_\beta^j:=h^{-N}\int_{x_\beta+R_h}U^{j}(x)\dd x=U^j(x)$ for all
$x\in x_\beta+R_h$. In particular, $U^{j}(x_\beta)=U_\beta^j$.
  
  Now, let $y\in
  x_\beta+R_h$ be such that the scheme \eqref{defNumSch} holds at
  $y$.
  Since the grid $\Grid$ is uniform and $\mathcal S_1,\mathcal
  S_2\subset \Grid$, $U^{j}(y+z_\beta)=U_{\beta+\gamma}^j$ for any
  $z_\gamma\in\Grid$ and any $j$, and then

 \begin{align*}
\Levy_i^h[\varphi_i^h(U^{j})](y)&=\sum_{\beta\not=0}
\big(\varphi_i^h(U^{j}(y+z_\beta))-\varphi_i^h(U^{j}(y))\big)
\omega_{\beta,h}\\
&=\sum_{\beta\not=0}
\big(\varphi_i^h(U_{\beta+\gamma}^{j})-\varphi_i^h(U_{\beta}^{j})\big)\omega_{\beta,h}=\Levy_i^h[\varphi_i^h(U^j_{\cdot})]_\beta
\end{align*}
where $\Levy_i^h$ on the left is understood as an operator on functions on $\Grid$.
%is the restriction to $L^\infty(\Grid; \dd \nu^{h,i})$ \normalcolor for $i=1,2$. 
Since $U^j$ satisfies \eqref{defNumSch} at $y$, we can
therefore conclude
that $U_{\beta}^{j}=U^j(y)$ satisfies \eqref{FullyDiscNumSch1} at $x_\beta$. 

\medskip
\noindent (b) Since $\ol{U^{j}}(y+z_\beta)=U_{\beta+\gamma}^j$ for any
  $z_\gamma\in\Grid$ and any $y\in x_\beta + R_h$ and the scheme
\eqref{FullyDiscNumSch1} holds at $x_\beta$, similar considerations as
in the proof of part (a) show that $\ol{U^{j}}$ satisfy the scheme
\eqref{defNumSch} at every point in $x_\beta+R_h$.
\end{proof}

\begin{proof}[Proof of Theorem \ref{thm:fullydisc}]  The equivalence
  given by Proposition \ref{prop:equivalence} ensures that parts
  \eqref{eu}--\eqref{thm:fullydisc:eqconttime}
  follow from the fact
  that $U_\beta^j$ (the solution of \eqref{FullyDiscNumSch1}) is the
  restriction to the grid $\Grid$ of $U_h^j$ (the solution of
  \eqref{defNumSch}). Integrals become sums because for functions
  $V$ on $\Grid$ with interpolants $\overline{V}$, 
\[
\int_{\R^N} \overline{V}(x) \dd x=h^N\sum_{\beta\neq0} V_\beta.
\]
\eqref{thm:fullydisc:conv} Let $U_h^j$ be the solution of
\eqref{defNumSch} for $u_0$ and $F^j$. Respectively let
$\overline{U^j}$ be the solution of \eqref{defNumSch} for
$\overline{U^0}$ and $\overline{F^j}(x)$. Then, for all $j\in\J$, by Theorem
\ref{thm:propertiesscheme:contractive} and continuity of $L^1$-translation,
\begin{equation*}
\begin{split}
\int_{\R^N} |\overline{U^j}(x)-U_h^j(x)|\dd x & \leq \int_{\R^N} |\overline{U^0}(x)-u_0(x)|\dd x+\sum_{l=1}^j \Delta t_l\int_{\R^N} |\overline{F^l}(x)-F^l(x)|\dd x\\
&\leq\lambda_{u_0,f}(h)\to 0 \qquad\textup{ as }\qquad h\to0^+.
\end{split}
\end{equation*}

Now for any compact $K\subset \R^N$,
\begin{equation*}%\label{convinterp2}
\begin{split}
  &\vertiii{U-u}_{K}=\max_{t_j\in\GridT}\|\overline{U^j}-u(\cdot,t_j)\|_{L^1(K)}\\
&\leq\max_{t_j\in\GridT}\|\overline{U^j}-U^j_h\|_{L^1(K)}+\max_{t_j\in\GridT}\|U^j_h-u(\cdot,t_j)\|_{L^1(K)}\\
&\leq \lambda_{u_0,f}(h)
  +\sup_{t\in[0,T]}\|\Ut_h(\cdot,t)-u(\cdot,t)\|_{L^1(K)}
  %\to 0 \qquad\textup{ as }\qquad h\to0^+.
\end{split}
\end{equation*}
which tends to zero as $h\to0^+$ by Theorem \ref{thm:conv}.
%In this way, given any compact set $K\subset\R^N$, we have that
%\begin{equation*}
%\begin{split}
% &\sum_{x_\beta\subset K} |h^N U_\beta^j-h^Nu(x_\beta,t_j)|=
% \sum_{x_\beta\subset K}
% \left|\int_{x_\beta+R_h}\big(\overline{U^j}(x)-u(x_\beta,t_j)\big)\dd x\right|\\
%&\leq \sum_{x_\beta\subset K} \int_{x_\beta+R_h}\left|\overline{U^j}(x) - u(x,t_j)\right|\dd x + \sum_{x_\beta\subset K}\int_{x_\beta+R_h}\left|u(x,t_j)-u(x_\beta,t_j)\right|\dd x
%\end{split}
%\end{equation*}
%The first integral in the last inequality goes to zero by \eqref{convinterp2} and the second by convergence of the Riemann sum (recall that $u\in L^1$).
\end{proof}

\subsection{A priori estimates for distributional solutions}
\label{sec:proofAprioriDistSol}

\begin{proof}[Proof of Proposition \ref{propconstructeddistsol}]
We will prove the results by passing to the limit in the a
priori estimates for $\Ut_h,\Vt_h$ in Theorems
\ref{thm:propertiesscheme} and \ref{thm:propertiesscheme:contractive}.
To do that we note that by Theorem \ref{thm:conv}, $\Ut_h,\Vt_h\to u,v$ in
$C([0,T];L_\textup{loc}^1(\R^N))$ and a.e. (for a subsequence) as
$h\to0^+$. We also observe that for $X=L^1(\R^N)$, $X=L^\infty(\R^N)$, and
$t\in[0,T-\Delta t]$, 
\begin{equation*}
\begin{split}
I=\bigg|\int_{0}^{t+\Delta t}\|f(\cdot,\tau)\|_X\dd \tau-\int_0^t\|f(\cdot,\tau)\|_X\dd \tau\bigg|
=\int_{t}^{t+\Delta t}\|f(\cdot,\tau)\|_X\dd \tau.
\end{split}
\end{equation*}
Since $\mathbf{1}_{(t,t+\Delta t]}(\tau)\to0$ a.e. as
$\Delta t\to0^+$ and \eqref{gas} hold, $I\to 0$  as
$\Delta t\to0^+$ by the dominated convergence theorem. Similar
results hold for the other time integrals that appear on the right
hand-sides in Theorems \ref{thm:propertiesscheme} and
\ref{thm:propertiesscheme:contractive}.

\medskip
\noindent\eqref{propconstructeddistsol:intbound} and
\eqref{propconstructeddistsol:contraction} then follow from Theorems
\ref{thm:propertiesscheme} \eqref{thm:propertiesscheme:intstable} and
\ref{thm:propertiesscheme:contractive} and Fatou's lemma. 
\medskip

\noindent \eqref{propconstructeddistsol:comp} is an immediate
consequence of \eqref{propconstructeddistsol:contraction}.
\medskip

\noindent \eqref{propconstructeddistsol:bndbound} follows from the
$L^\infty$-bound Theorem \ref{thm:propertiesscheme}
\eqref{thm:propertiesscheme:bndstable}, the estimate
$|u|\leq|u-\Ut_h|+|\Ut_h|$, and the a.e. convergence of $\Ut_h$.

\medskip
\noindent\eqref{propconstructeddistsol:timereg} follows by the triangle
inequality, Theorem \ref{thm:timereg} (see also Lemma \ref{lem:timereginterpolant}), and passing to the limit:
\begin{equation*}
\begin{split}
\|u(\cdot,t)-u(\cdot,s)\|_{L^1(K)}&\leq2\|u-\Ut_h\|_{C([0,T];L^1(K))}+\Lambda_{K}(|t-s|)\\
&\quad+|K|\int_{s}^{t}\|f(\cdot,\tau)\|_{L^\infty(\R^N)}\dd \tau+|K|\,\lambda(|t-s|, \Delta t),
\end{split}
\end{equation*}
where $\|u-\Ut_h\|_{C([0,T];L^1(K))}$ and $\lambda(|t-s|, \Delta t)$ goes to zero when $h\to0^+$.
%Which completes the proof.
\end{proof}

%%%%%%%%%%%%%%%%%%%%%%%%%%%%%%%%%%%%%%%%%%%%%%%%%%%%
%%%%%%%%%%%%%%%%%%%%%NEW SECTION%%%%%%%%%%%%%%%%%%%%%%%
%%%%%%%%%%%%%%%%%%%%%%%%%%%%%%%%%%%%%%%%%%%%%%%%%%%%

\section{Auxiliary results}
\label{sec:auxresults}

\subsection{The operator $\Te$}
Theorem \ref{propOpT:proplimsolnofEllipP} with $T=\Te$
follows from the three results of this section. Note that we do not
need $\psi\in L^1$ in most of the results.  

\begin{lemma}\label{MonoT}
Assume \eqref{nuas}, \eqref{philipas}, $\Lvar
\nu(\R^N)\leq1$  and
$\psi, \hat{\psi} \in L^\infty(\R^N)$. 
If $\psi\leq \hat{\psi}$ a.e., then $\Te[\psi]\leq
\Te[\hat{\psi}]$ a.e. 
\end{lemma}

\begin{proof}
By definition
\begin{equation*}
\begin{split}
&\Te[\psi](x)-\Te[\hat{\psi}](x)\\
&=\psi(x)-\hat{\psi}(x) + \int_{\R^N}\Big(\big(\varphi\left(\psi\right)-\varphi(\hat{\psi})\big)(x+z)-\big(\varphi\left(\psi\right)-\varphi(\hat{\psi})\big)(x)\Big)\dd \nu(z).\\
\end{split}
\end{equation*}
Since $\varphi$ is nondecreasing and $\psi\leq \hat{\psi} $,
$\varphi(\psi)-\varphi(\hat{\psi})\leq0$ and 
\begin{equation*}
\begin{split}
\Te[\psi](x)-\Te[\hat{\psi}](x)&\leq \psi(x)-\hat{\psi}(x) +0-\big(\varphi\left(\psi\right)-\varphi(\hat{\psi})\big)(x)\int_{\R^N}\dd \nu(z).%\\
\end{split}
\end{equation*}
By \eqref{philipas} and the mean value theorem there exists $\xi \in [0, \Lvar]$ such that $ \varphi\left(\psi(x)\right)-\varphi(\hat{\psi}(x))=\xi \big( \psi(x)-\hat{\psi}(x)\big)$. Hence, 
\begin{equation*}
\Te[\psi](x)-\Te[\hat{\psi}](x)\leq \big(\psi(x)-\hat{\psi}(x)\big)\left[1-\xi   \nu(\R^N)\right].
\end{equation*}
Hence $\Te[\psi]-\Te[\hat\psi]\leq 0$ since $\psi-\hat\psi\leq0$ and $\xi
\nu(\R^N)\leq \Lvar \nu(\R^N)\leq1$.
\end{proof}

Now we deduce an $L^1$-contraction result for $\Te$.

\begin{lemma}\label{TContT}
Assume \eqref{nuas}, \eqref{philipas}, $\Lvar \nu(\R^N)\leq1$, and
$\psi, \hat{\psi} \in L^\infty(\R^N)$, and $(\psi-\hat{\psi})^+ \in
L^1(\R^N)$.  Then 
\[
\int_{\R^N}(\Te[\psi](x)-\Te[\hat{\psi}](x))^+\dd x\leq\int_{\R^N}(\psi(x)-\hat{\psi}(x))^+\dd x.
\]
\end{lemma}

\begin{proof}
This result follows as in the so-called Crandall-Tartar lemma, see e.g.
Lemma 2.12 in \cite{HoRi02}. We include the argument for
completeness. Since $\psi\vee
  \hat{\psi}\in L^\infty(\R^N)$ and $\psi\leq \psi\vee
  \hat{\psi}$, 
  we have by Lemma \ref{MonoT} that $\Te[\psi]\leq
  \Te[\psi\vee \hat{\psi}]$ and
  $\Te[\psi]-\Te[\hat{\psi}]\leq \Te[\psi\vee
  \hat{\psi}]-\Te[\hat{\psi}]$. Moreover, since $\hat{\psi}\leq \psi\vee
  \hat{\psi}$, we have by Lemma \ref{MonoT} again that
   $0=\Te[\hat{\psi}]-\Te[\hat{\psi}]\leq \Te[\psi\vee
  \hat{\psi}]-\Te[\hat{\psi}]$. Hence,
  $$(\Te[\psi]-\Te[\hat{\psi}])^+\leq \Te[\psi\vee
  \hat{\psi}]-\Te[\hat{\psi}],$$
and
\begin{equation*}
\begin{split}
\Te[\psi\vee\hat{\psi}](x)-\Te[\hat{\psi}](x) &= (\psi\vee \hat{\psi}-\hat{\psi}) 
+ \Levynu[\varphi(\psi\vee
 \hat{\psi})-\varphi(\hat{\psi})].
\end{split}
\end{equation*}
Next note that by \eqref{philipas}, 
$$0\leq \varphi(\psi\vee
 \hat{\psi})-\varphi(\hat{\psi})\leq \Lvar( \psi\vee
 \hat{\psi}-\hat{\psi})=\Lvar(\psi-\hat{\psi})^+ \in L^1(\R^N),$$
and hence, since $\Te$ is conservative by Lemma \ref{Levynuwell-def} (b),
\begin{equation*}
\begin{split}
&\int_{\R^N}(\Te[\psi]-\Te[\hat{\psi}])^+\dd x\\
&\leq\int_{\R^N} \big(\Te[\psi\vee\hat{\psi}]-\Te[\hat{\psi}]\big)\dd x=\int_{\R^N} \big( (\psi\vee \hat{\psi})-\hat{\psi} \big)\dd x=\int_{\R^N} (\psi-\hat{\psi})^+ \dd x,
\end{split}
\end{equation*}
which completes the proof.
\end{proof}

\begin{corollary}
Assume \eqref{nuas}, \eqref{philipas}, $\Lvar \nu(\R^N)\leq1$, and $\psi\in L^\infty(\R^N)$. Then
$$
\|\Te[\psi]\|_{L^\infty(\R^N)}\leq \|\psi\|_{L^\infty(\R^N)}.
$$
If also $\psi\in L^1(\R^N)$, then 
$$
\|\Te[\psi]\|_{L^1(\R^N)}\leq \|\psi\|_{L^1(\R^N)}.
$$
\end{corollary}

\begin{proof} 
The case $p=1$ is just a direct consequence of Lemma \ref{TContT}. For $p=\infty$, note that 
$\Te[\|\psi\|_{L^\infty(\R^N)}]=\|\psi\|_{L^\infty(\R^N)}
$
and
$
\Te[-\|\psi\|_{L^\infty(\R^N)}]=-\|\psi\|_{L^\infty(\R^N)}.
$
Since 
$$
-\|\psi\|_{L^\infty(\R^N)}\leq \psi\leq \|\psi\|_{L^\infty(\R^N)},
$$ 
we conclude by Lemma \ref{MonoT} that $-\big(\|\psi\|_{L^\infty(\R^N)}\big)\leq \Te[\psi]\leq \|\psi\|_{L^\infty(\R^N)}$.
\end{proof}

\subsection{The operator $\Ti$}

Now we prove Theorem \ref{maincor} and Theorem
  \ref{propOpT:proplimsolnofEllipP} with $T=\Ti$. We start by a
  uniqueness result for bounded distributional solutions of
\begin{equation}\label{NAgenEllipP}\tag{Gen-EP}
w(x)-\Operator[\varphi(w)](x)=\rho(x)\qquad x\in \R^N.
\end{equation}

\begin{theorem}[Uniqueness, Theorem 3.1 in \cite{DTEnJa17b}]\label{NAuniquemuEP}
Assume \eqref{phias}, \eqref{muas} and $\rho\in L^\infty(\R^N)$. Then there is at most one distributional solution $w$ of \eqref{NAgenEllipP} such that $w \in L^\infty(\R^N)$ and $w-\rho\in L^1(\R^N)$.
\end{theorem}

From now on we restrict ourselves to \eqref{EllipP}
  which is a special case of \eqref{NAgenEllipP}.
By approximation, stability, and compactness results, we will prove that constructed solutions of \eqref{EllipP} indeed satisfy Theorem \ref{NAuniquemuEP}, and hence, we obtain existence and a priori results.
Let us start by a contraction principle for globally
  Lipschitz $\varphi$'s, a more general result will
be given later.

\begin{lemma}\label{TContEllip}
Assume \eqref{nuas}, $\varphi:\R\to\R$ is nondecreasing and globally Lipschitz, and $(w-\hat{w})^+,
(\rho-\hat{\rho})^+\in L^1(\R^N)$. If $w, \hat{w}$
are respective a.e. sub- and supersolutions of \eqref{EllipP} with right-hand sides $\rho, \hat{\rho}$, then 
\[
\int_{\R^N}(w(x)-\hat{w}(x))^+\dd x\leq\int_{\R^N}(\rho(x)-\hat{\rho}(x))^+\dd x.
\]
\end{lemma}

\begin{proof}
Subtract the equations for $w$ and $\hat{w}$ and multiply by $\text{sign}^+(w-\hat{w})$ to get
\[
(w-\hat{w})\text{sign}^+(w-\hat{w}) \leq (\rho-\hat{\rho})\text{sign}^+(w-\hat{w})+\Levynu [\varphi(w)-\varphi(\hat{w})]\text{sign}^+(w-\hat{w})
\]
Note that $(w-\hat{w})\text{sign}^+(w-\hat{w})=(w-\hat{w})^+$, $(\rho-\hat{\rho})\text{sign}^+(w-\hat{w})\leq (\rho-\hat{\rho})^+$ and $\Levynu [\varphi(w)-\varphi(\hat{w})]\text{sign}^+(w-\hat{w})\leq\Levynu [(\varphi(w)-\varphi(\hat{w}))^+]$. The latter is an example of a standard convex inequality, see e.g. page 149 in \cite{Ali07}.   Thus,
\begin{equation*}\label{nonlinearEllipPIneq}
(w-\hat{w})^+\leq(\rho-\hat{\rho})^+ +\Levynu [(\varphi(w)-\varphi(\hat{w}))^+].
\end{equation*}
The assumption on $\varphi$ ensures that $(\varphi(w)-\varphi(\hat{w}))^+ \in L^1(\R^N)$. Indeed, for the global Lipschitz constant $L_\varphi$, and with $\Omega_+:=\{x\in\R^N: w(x)>\hat{w}(x)\}$, we have
\[
\begin{split}
&\int_{\R^N}(\varphi(w(x))-\varphi(\hat{w}(x)))^+\dd x= \int_{\Omega_+} \big(\varphi(w(x))-\varphi(\hat{w}(x))\big)\dd x\\
& \leq L_\varphi  \int_{\Omega_+}\big( w(x)-\hat{w}(x)\big)\dd x = L_\varphi  \int_{\R^N} (w(x)-\hat{w}(x))^+\dd x.
\end{split}
\]
Thus, we integrate over $\R^N$ and use Lemma \ref{Levynuwell-def} (b) to get
\begin{equation*}
\begin{split}
\int_{\R^N}(w(x)-\hat{w}(x))^+\dd x&\leq\int_{\R^N}(\rho(x)-\hat{\rho}(x))^+\dd x + \int_{\R^N}\Levynu [(\varphi(w)-\varphi(\hat{w}))^+](x)\dd x\\
&=\int_{\R^N}(\rho(x)-\hat{\rho}(x))^+\dd x
\end{split}
\end{equation*}
which completes the proof.
\end{proof}

Here are some standard consequences of the contraction result.

\begin{corollary}[A priori estimates]\label{propEllipP}
Assume \eqref{nuas}, $\varphi:\R\to\R$ is nondecreasing and globally Lipschitz, and $w,\hat{w}, \rho,\hat{\rho}\in L^1(\R^N)$. If $w, \hat{w}$
are respective a.e. sub- and supersolutions of \eqref{EllipP} with right-hand sides $\rho, \hat{\rho}$, then
\begin{enumerate}[{\rm (a)}]
\item  \textup{($L^1$-contraction)}
$\int_{\R^N}\big(w(x)-\hat{w}(x)\big)^+\dd x \leq\int_{\R^N}\big(\rho(x)-\hat{\rho}(x)\big)^+\dd x$,
\item \textup{(Comparison)} if $\rho\leq \hat{\rho}$ a.e., then $w\leq \hat{w}$ a.e., and
\item \textup{($L^1$-bound)} $\|w\|_{L^1(\R^N)}\leq \|\rho\|_{L^1(\R^N)}$.
\end{enumerate}
\end{corollary}

\begin{lemma}[A priori estimate, $L^\infty$-bound]\label{propEllipPbounded}
Assume \eqref{nuas}, \eqref{philipas}, and $w,\rho\in L^\infty(\R^N)$. If  $w$ solves \eqref{EllipP} a.e. with right-hand side $\rho$ respectively, then
$$
\|w\|_{L^\infty(\R^N)}\leq\|\rho\|_{L^\infty(\R^N)}.
$$
\end{lemma}

\begin{proof}
Since $w\in L^\infty(\R^N)$, for every $\delta>0$, there exists
$x_\delta\in \R^N$ such that
$$
w(x_\delta)+\delta>\esssup_{x\in\R^N}\{w(x)\}.
$$
I.e. $|\esssup w-w(x_\delta)|
<\delta$, and then by \eqref{philipas},
\begin{equation*}
\begin{split}
  \varphi\big(\esssup_{x\in\R^N}\{w(x)\}\big)-\varphi\big(w(x_\delta)\big)
 &\leq L_\varphi\big|\esssup_{x\in\R^N}\{w(x)\}-w(x_\delta)\big|<L_\varphi\delta.
\end{split}
\end{equation*}
Combining the above and \eqref{philipas}  and \eqref{nuas}, we get
\begin{equation*}
\begin{split}
&\esssup_{x\in\R^N}\{w(x)\}-\delta-\rho(x_\delta)<w(x_{\delta})-\rho(x_\delta)=\Levy^\nu[\varphi(w(\cdot))](x_\delta)\\
&\leq \int_{\R^N}\Big(\varphi\big(\esssup_{x\in\R^N}\{w(x)\}\big)-\varphi(w(x_\delta))\Big)\dd \nu(z)<L_\varphi\delta\nu(\R^N),
\end{split}
\end{equation*}
and hence,
$$
\esssup_{x\in\R^N}\{w(x)\}
<\|\rho^+\|_{L^\infty(\R^N)}+\delta(1+L_\varphi\nu(\R^N)).
$$
We may send $\delta$ to zero to get
$$
\|w^+\|_{L^\infty(\R^N)}=(\esssup_{x\in\R^N}\{w(x)\})^+\leq\|\rho^+\|_{L^\infty(\R^N)}.
$$
In a similar way $\esssup_{x\in\R^N}\{-w(x)\}\leq\|\rho^-\|_{L^\infty(\R^N)}$, and the result follows.
\end{proof}

Under stronger assumptions on $\varphi$ we now establish an existence
result for \eqref{EllipP} in $L^1\cap L^\infty$. By this result and an
approximation argument, we get the general existence result
which holds under assumption \eqref{phias}. As a consequence of the
approximation argument, the general problem will also inherit the a priori estimates
in Corollary \ref{propEllipP} and Lemma \ref{propEllipPbounded}.

\begin{proposition}\label{firstex}
Assume \eqref{nuas}, 
$\rho \in L^1(\R^N)$, and
 \begin{equation}\label{strictassumptionphi}
 \varphi \in C^1(\R)\text{ such that }\frac{1}{c}\leq \varphi'(s)\leq c\text{ for all $s\in \R$ and some $c>1$.}
 \end{equation}
Then there exists a unique $w\in L^1(\R^N)$ satisfying \eqref{EllipP} a.e. Moreover, if in addition $\rho\in L^\infty(\R^N)$, then $w$ is also in $L^\infty(\R^N)$.
\end{proposition}

\begin{remark}\label{rem:firstex}
Let $\rho\in L^1(\R^N)\cap L^\infty(\R^N)$. By Lemma \ref{propEllipPbounded}, we can, a posteriori, obtain the above existence and uniqueness result for the less restrictive assumption
\begin{equation*}
 \varphi \in C^1(\R)\text{ such that }\frac{1}{c}\leq \varphi'(s)\leq c\text{ for all $s\in K$ and some $c>1$,}
 \end{equation*}
where $K\subset \R$ is the compact set $\{\xi\in\R \,:\, |\xi|\leq\|\rho\|_{L^\infty(\R^N)}\}$.
\end{remark}

\begin{proof}
By \eqref{nuas}, equation \eqref{EllipP} can be written in an expanded way as follows:
\begin{equation}\label{eq:existbis}
w(x)+\nu(\R^N) \varphi(w(x))=\int_{\R^N}\varphi(w(x+z))\dd\nu(z)+\rho(x).
\end{equation}
Define 
$$\W(x):=\Phi(w(x)):= w(x)+\nu(\R^N) \varphi(w(x)),$$
and note that by assumptions, $\Phi(0)=0+\nu(\R^N) \varphi(0)=0$, $\Phi \in
C^1(\R)$ is invertible, 
\(
1+ \nu(\R^N)\frac{1}{c}\leq \Phi'\leq1+\nu(\R^N)c
\), and the inverse $\Phi^{-1}\in C^1(\R)$ satisfies
\begin{equation}\label{Phiinvbdd}
\frac{1}{1+\nu(\R^N)c} \leq (\Phi^{-1})'(s)\leq \frac{1}{1+\nu(\R^N)\frac{1}{c}}\qquad\text{for all}\qquad s\in \R.
\end{equation}
Since $\psi_1:=w-\hat{w}$ and $\psi_2:=\varphi(w)-\varphi(\hat{w})$
have the same sign, $|\psi_1+\psi_2|=|\psi_1|+|\psi_2|$, and thus
\begin{equation}\label{L1relation}
\begin{split}
  \|\W-\hat{\W}\|_{L^1(\R^N)}
&=\|w-\hat{w}\|_{L^1(\R^N)}+ \nu(\R^N)\|\varphi(w)-\varphi(\hat{w})\|_{L^1(\R^N)}.
\end{split}
\end{equation}

With all the mentioned properties of $\Phi$ we are allowed to write equation \eqref{eq:existbis} in terms of $\W$ and $\Phi$ in the following way: 
\begin{equation}\label{EPV}
\W(x)=\int_{\R^N}\varphi\left(\Phi^{-1}\left(\W(x+z)\right)\right)\dd\nu(z)+\rho(x).
\end{equation}
To conclude, we will prove that the map defined by
\begin{equation*}
\W\mapsto \mathbf{M}[\W]:=\int_{\R^N}\varphi\left(\Phi^{-1}\left(\W(x+z)\right)\right)\dd\nu(z)+\rho(x),
\end{equation*}
is a contraction in $L^1(\R^N)$. In this way, Banach's fixed point
theorem will ensure the existence of a unique solution $\W\in
L^1(\R^N)$ of \eqref{EPV}, and thus, the existence of a unique
solution $w\in L^1(\R^N)$ of \eqref{EllipP} by the invertibility of
$\Phi$. Indeed, using first the definition of $\Phi$ and \eqref{L1relation}
and then \eqref{Phiinvbdd}, we have

\begin{equation*}
\begin{split}
\|\mathbf{M}[\W]-\mathbf{M}[\hat{\W}]\|_{L^1(\R^N)}&\leq \nu(\R^N) \| \varphi(\Phi^{-1}(\W))-\varphi(\Phi^{-1}(\hat{\W}))\|_{L^1(\R^N)}\\
&=\|\W-\hat{\W}\|_{L^1(\R^N)}-\|\Phi^{-1}(\W)-\Phi^{-1}(\hat{\W})\|_{L^1(\R^N)}\\
&\leq\|\W-\hat{\W}\|_{L^1(\R^N)}-\min_{s\in \R}|(\Phi^{-1})'(s)|\|\W-\hat{\W}\|_{L^1(\R^N)}\\
&=\left(1-\frac{1}{1+\nu(\R^N) c}\right)\|\W-\hat{\W}\|_{L^1(\R^N)}.\\
%&<\|\W-\hat{\W}\|_{L^1(\R^N)}.
\end{split}
\end{equation*}

Let us now consider the case when $\rho\in L^1(\R^N)\cap L^\infty(\R^N)$. By \eqref{strictassumptionphi}, there exists a unique $\varphi^{-1}$ such that
\begin{equation}\label{varphiinversederivativebounded}
\frac{1}{c}\leq (\varphi^{-1})'(s)\leq c\qquad\text{for all}\qquad s\in\R.
\end{equation}
Now, define $W(x):=\varphi(w(x))$ which is (only) in $L^1(\R^N)$ since $w$ is, and it solves
$$
\varphi^{-1}(W(x))-\Levy^{\nu}[W](x)=\rho(x)\qquad\text{a.e.}\qquad x\in\R^N.
$$
Note that \eqref{varphiinversederivativebounded} means that
$$
\frac{1}{c}s\leq \varphi^{-1}(s)\leq cs\qquad\text{for all}\qquad s\geq0,
$$
and
$$
cs\leq \varphi^{-1}(s)\leq \frac{1}{c}s\qquad\text{for all}\qquad s\leq0.
$$
Therefore we also consider $Q,R\in L^1(\R^N)$ solving
$$
\varphi^{-1}(Q(x))-\Levy^{\nu}[Q](x)=\rho^+(x)\qquad\text{and}\qquad \varphi^{-1}(R(x))-\Levy^{\nu}[R](x)=\rho^-(x).
$$
By Corollary \ref{propEllipP} (b), we immediately have that $Q\geq0$, $R\leq0$, and $R\leq W\leq Q$. Under these considerations,
$$
\frac{1}{c}Q(x)-\Levy^{\nu}[Q](x)\leq \rho^+(x)\qquad\text{and}\qquad \frac{1}{c}R(x)-\Levy^{\nu}[R](x)\geq\rho^-(x).
$$
%The properties of $\varphi^{-1}$ then ensure that $W$ also solves the linear elliptic inequalities
%\begin{equation*}
%\frac{1}{c}W(x)-\Levy^{\nu}[W](x)\leq \rho(x)\qquad\text{and}\qquad cW(x)-\Levy^{\nu}[W](x)\geq \rho(x).
%\end{equation*}
By Theorem 3.1 (b) and (c) in \cite{DTEnJa17a}, there exist unique
a.e.-solutions $q,r\in L^1(\R^N)\cap L^\infty(\R^N)$ of 
\begin{equation*}
\frac{1}{c}q(x)-\Levy^{\nu}[q](x)= \rho^+(x)
\qquad\text{and}\qquad
\frac{1}{c}r(x)-\Levy^{\nu}[r](x)= \rho^-(x)
\end{equation*}
which satisfy
$$
\|q\|_{L^\infty(\R^N)}\leq c\|\rho^+\|_{L^\infty(\R^N)}\leq c\|\rho\|_{L^\infty(\R^N)}
$$
and
$$
\|r\|_{L^\infty(\R^N)}\leq c\|\rho^-\|_{L^\infty(\R^N)}\leq c\|\rho\|_{L^\infty(\R^N)}.
$$
Lemma \ref{TContEllip} then gives $Q\leq q$ and $r\leq R$.
%Since $\frac{1}{c}\big(Q(x)-q(x)\big)-\Levy^{\nu}[Q-q](x)\leq
%0$, $Q\leq q$ a.e. by the argument of the proof of Lemma
%\ref{TContEllip}. Similarly, we also have that $r\leq R$
%a.e. 
These estimates and the definition of $W$ yield
$$
-c\|\rho\|_{L^\infty(\R^N)}\leq r\leq R\leq \varphi(w(x))\leq Q\leq q\leq c\|\rho\|_{L^\infty(\R^N)}.
$$
Finally, by \eqref{varphiinversederivativebounded}, we then get
$$
-c^2\|\rho\|_{L^\infty(\R^N)}\leq w\leq c^2\|\rho\|_{L^\infty(\R^N)}.
$$
The proof is complete.
%
%The properties of $\varphi^{-1}$ then ensure that $W$ also solves the linear elliptic inqualities
%\begin{equation*}
%\frac{1}{c}W(x)-\Levy^{\nu}[W](x)\leq \rho(x)\qquad\text{and}\qquad cW(x)-\Levy^{\nu}[W](x)\geq \rho(x).
%\end{equation*}
%By Theorem 3.1 (b) and (c) in \cite{DTEnJa17a}, there exist unique
%a.e.-solutions $u_{\frac{1}{c}},u_{c}\in L^1(\R^N)\cap L^\infty(\R^N)$ of 
%\begin{equation*}
%\frac{1}{c}u_{\frac{1}{c}}(x)-\Levy^{\nu}[u_{\frac{1}{c}}](x)= \rho(x)
%\qquad\text{and}\qquad
%cu_{c}(x)-\Levy^{\nu}[u_{c}](x)= \rho(x).
%\end{equation*}
%that satisfy
%$$\|u_{\frac{1}{c}}\|_{L^\infty(\R^N)}\leq
%c\|\rho\|_{L^\infty(\R^N)}\qquad\text{and}\qquad\|u_{c}\|_{L^\infty(\R^N)}\leq \frac{1}{c}\|\rho\|_{L^\infty(\R^N)}.$$
%Since 
%$\frac{1}{c}\big(W(x)-u_{\frac{1}{c}}(x)\big)-\Levy^{\nu}[W(x)-u_{\frac{1}{c}}(x)]\leq
%0$, $W\leq u_{\frac{1}{c}}$ a.e. by the argument of the proof of Lemma
%\ref{TContEllip}. Similarly, we also have that $u^c\leq W$
%a.e.. By these estimates and the definition of $W$,
%we find that
%$$
%-\frac{1}{c}\|\rho\|_{L^\infty(\R^N)}\leq \varphi(w(x))\leq c\|\rho\|_{L^\infty(\R^N)}.
%$$
%Finally, by \eqref{varphiinversederivativebounded}, we then get
%$$
%-\frac{1}{c^2}\|\rho\|_{L^\infty(\R^N)}\leq w\leq c^2\|\rho\|_{L^\infty(\R^N)}.
%$$
%The proof is complete.
\end{proof}

\begin{proof}[Proof of Theorem \ref{maincor}]
The proof is divided into four steps.

\medskip
\noindent{\bf 1)} {\it Approximate problem.} For
$\delta>0$, let $\omega_\delta$ be a standard mollifier and define  
\[
\varphi_\delta(\zeta):=
(\varphi*\omega_\delta)(\zeta)-(\varphi*\omega_\delta)(0)+ \delta \zeta.
\]
The properties of mollifiers give $\varphi_\delta\in C^\infty(\R)$, and hence, it is locally Lipschitz. 
Moreover,
$\varphi_\delta'\geq\delta>0$ and $\varphi_\delta(0)=0$. Then there exists a constant $c>1$
such that, for every compact set $K\subset\R$,
$ \frac{1}{c}\leq\varphi_\delta'(s)\leq c$ for all $s\in K$. 
By Proposition \ref{firstex} and Remark \ref{rem:firstex}, there exists a unique a.e.-solution $w_\delta\in L^1(\R^N)\cap L^\infty(\R^N)$ of 
\begin{equation}\label{ellipap}
w_\delta(x)- \Levynu[\varphi_\delta(w_\delta)](x)=\rho(x) \qquad\text{for all}\qquad x\in \R^N,
\end{equation} 
and moreover, by Corollary \ref{propEllipP} (c) and Lemma \ref{propEllipPbounded},
\begin{equation}\label{propapproxsoln}
\|w_\delta\|_{L^1(\R^N)}\leq\|\rho\|_{L^1(\R^N)}\qquad\text{and}\qquad\|w_\delta\|_{L^\infty(\R^N)}\leq\|\rho\|_{L^\infty(\R^N)}
\end{equation}

\medskip
\noindent{\bf 2)} {\it $L^1_{\textup{loc}}$-converging subsequence with limit $w$.} Let $K\subset\R^N$ be compact and
  $w_\delta^K(x):=w_\delta(x)\mathbf{1}_K(x)$ for any $\delta>0$. We then
apply Kolmogorov-Riesz's compactness theorem (cf. e.g. Theorem A.5 in
\cite{HoRi02}). First, by
\eqref{propapproxsoln},
$\|w_\delta^K\|_{L^1(\R^N)}\leq\|w_\delta\|_{L^1(\R^N)}\leq\|\rho\|_{L^1(\R^N)}$. Second,
note that $w_\delta(\cdot +\xi)$ is a solution of \eqref{ellipap} with
right-hand side $\rho(\cdot+\xi)$, and then, by  Corollary
\ref{propEllipP} (a) and \eqref{propapproxsoln} again and since
translations are continuous in $L^1(\R^N)$,
\begin{equation*}
\begin{split}
&\|w_\delta^K(\cdot+\xi)-w_\delta^K\|_{L^1(\R^N)}\\
&\leq \|(w_\delta(\cdot+\xi)-w_\delta)\mathbf{1}_K(\cdot+\xi)\|_{L^1(\R^N)}+\|w_\delta(\mathbf{1}_K(\cdot+\xi)-\mathbf{1}_K)\|_{L^1(\R^N)}\\
&\leq\|\rho(\cdot+\xi)-\rho\|_{L^1(\R^N)}+\|\rho\|_{L^\infty(\R^N)}\|\mathbf{1}_K(\cdot+\xi)-\mathbf{1}_K\|_{L^1(\R^N)}\to0\quad\text{as}\quad|\xi|\to0.
\end{split}
\end{equation*}
Hence, there exists $w\in L^1(K)$ and a subsequence
$\delta_n\to 0^+$ such that $w_{\delta_n} \to w$  in $L^1(K)$ as $n\to \infty$. A covering and diagonal argument then allow us to pick a further subsequence such that the convergence is in $L_\textup{loc}^1(\R^N)$, and hence, also pointwise a.e. Then, $w\in L^1(\R^N)\cap L^\infty(\R^N)$ since the estimates
\begin{equation}\label{proplimitsoln}
\|w\|_{L^1(\R^N)}\leq\|\rho\|_{L^1(\R^N)}\qquad\text{and}\qquad\|w\|_{L^\infty(\R^N)}\leq\|\rho\|_{L^\infty(\R^N)}
\end{equation}
hold by taking the a.e. limit using Fatou's lemma and the inequality $|w|\leq |w-w_{\delta_n}|+|w_{\delta_n}|$ respectiely in \eqref{propapproxsoln}.

\medskip
\noindent {\bf 3)} {\it The limit $w$ solves
    \eqref{EllipP} a.e.} 
Note that ($\varphi(0)=0$)
\begin{equation*}
  |\varphi_\delta(\zeta)-\varphi(\zeta)|
\leq |\varphi*\omega_\delta-\varphi
  |(\zeta)+|\varphi*\omega_\delta-\varphi|(0)+ \delta |\zeta|, 
\end{equation*}
which implies that $\varphi_\delta\to \varphi$ as $\delta \to 0^+$
locally uniformly by \eqref{phias} and properties of mollifiers.
Then by a.e.-convergence of
$w_{\delta_n}$, continuity of $\varphi$, and $\|w_{\delta_n}\|_{L^\infty}\leq \|f\|_{L^\infty}$, 
\begin{equation*}
\begin{split}
  |\varphi_{\delta_n}(w_{\delta_n})- \varphi(w)|
 &\leq
  \sup_{|\zeta|\leq\|\rho\|_{L^\infty}}|\varphi_{\delta_n}(\zeta)-
  \varphi(\zeta)|+| \varphi(w_{\delta_n})-\varphi(w)|\to 0
\end{split}
\end{equation*}
pointwise a.e. as $n\to\infty$.
Moreover, $|\varphi_{\delta_n}(w_{\delta_n})|\leq
|\varphi_{\delta_n}(w_{\delta_n})-\varphi(w_{\delta_n})|+|\varphi(w_{\delta_n})|$,
so for $n$ sufficiently large,
\begin{equation*}
\|\varphi_{\delta_n}(w_{\delta_n})\|_{L^\infty(\R^N)}\leq\sup_{|\zeta|\leq\|\rho\|_{L^\infty}}|\varphi(\zeta)|+1.
\end{equation*}
Then by the dominated convergence theorem and \eqref{nuas}, $\Levynu
[\varphi_{\delta_n}(w_{\delta_n})]\to \Levynu [\varphi(w)]$ pointwise
a.e. as $n\to\infty$. Hence we may pass to the limit in \eqref{ellipap} to
see that $w$ is an a.e.-solution of \eqref{EllipP}.

\medskip
\noindent{\bf 4)} {\it Uniqueness.}  
By the assumptions and \eqref{proplimitsoln}, $w,\rho\in L^1(\R^N)\cap
L^\infty(\R^N)$ and hence $w-\rho\in L^1(\R^N)$. Next we multiply equation
\eqref{EllipP}, satisfied a.e. by $w$, by 
a test function and integrate. Since $\Levynu$ is self-adjoint ($\nu$
is symmetric),
$$
\int_{\R^N}\Big(w\psi-\varphi(w)\Levynu[\psi]-\rho\psi\Big)\dd x =0\qquad\text{for all}\qquad \psi\in C_\textup{c}^\infty(\R^N).
$$
Hence, $w$ is a distributional solution of \eqref{EllipP}. By Theorem
\ref{NAuniquemuEP} it is then unique.
\end{proof}

\begin{proof}[Proof of Theorem \ref{propOpT:proplimsolnofEllipP} with $T=\Ti$]
By the proof of Theorem \ref{maincor}, we know that a.e.-solutions $w_\delta, \hat{w}_{\delta}$ of \eqref{ellipap} with respective right-hand sides $\rho,\hat{\rho}$ satisfy Corollary \ref{propEllipP} and Lemma \ref{propEllipPbounded}, and they converge a.e. to $w,\hat{w}$ which are solutions of \eqref{EllipP} with respective right-hand sides $\rho,\hat{\rho}$. Thus, we inherit (b) and (c) by Fatou's lemma, by the inequality $|w|\leq |w-w_{\delta}|+|w_{\delta}|$ and the a.e.-convergence we obtain (d), and (a) can be deduced from the $L^1$-contraction.
\end{proof}

\begin{remark}
By stability and compactness results for \eqref{EllipP}, we can get existence and a priori estimates for the full elliptic problem \eqref{NAgenEllipP}.
\end{remark}

\section*{Acknowledgements}
%\addcontentsline{toc}{section}{Acknowledgements}

The authors were supported by the Toppforsk (research excellence)
project Waves and Nonlinear Phenomena (WaNP), grant no. 250070 from
the Research Council of Norway. F.~del Teso was also supported by the
BERC 2014--2017 program from the Basque Government, BCAM Severo Ochoa
excellence accreditation SEV-2013-0323 from Spanish Ministry of
Economy and Competitiveness (MINECO), and the ERCIM  ``Alain
Benoussan'' Fellowship programme. We would like to
thank the referees for many good questions, remarks, and suggestions,
which have helped us improve the paper.

%%%%%%%%%%%%%%%%%%%%%%%%%%%%%%%%%%%%%%%%%%%%%%%%%%%%%%%%%%%%%%

%\lhead{\emph{References}} % Change the page header to say "References"

%\lhead{\emph{Bibliography}} % Change the page header to say "Bibliography"

%\bibliographystyle{acm} % Use the "unsrtnat" BibTeX style for formatting the Bibliography

%%\bibliographystyle{abbrv}

%If bibliography file is in the same folder use:

%\bibliography{Bibliography.bib} % The references (bibliography) information are stored in the file named "Bibliography.bib"

%If bibliography file is in a different folder use:

%\bibliography{../../../../../../Bibliography}  % The references (bibliography) information are stored in the file named "Bibliography.bib"

\end{document}